\documentclass{article}

\usepackage[a4paper, total={7in, 8in}]{geometry}
\usepackage[utf8]{inputenc}
\usepackage[english]{babel}
\usepackage{amsmath, amsthm}
\usepackage{amsfonts}
\usepackage{amssymb}
\usepackage{mathrsfs}
\usepackage{mathtools}
\mathtoolsset{showonlyrefs=true}
\usepackage{textcomp}
\usepackage{graphicx}
\usepackage{caption}
\usepackage{subcaption}
\usepackage{float}
\usepackage{color}
\usepackage{framed}
\usepackage{fancybox}
\usepackage{dsfont}
\usepackage{multicol}
\usepackage{comment}
\usepackage{enumerate}
\usepackage[usenames,dvipsnames]{xcolor}
\usepackage{relsize} 

\makeatletter
\let\save@mathaccent\mathaccent
\newcommand*\if@single[3]{%
  \setbox0\hbox{${\mathaccent"0362{#1}}^H$}%
  \setbox2\hbox{${\mathaccent"0362{\kern0pt#1}}^H$}%
  \ifdim\ht0=\ht2 #3\else #2\fi
  }
\newcommand*\rel@kern[1]{\kern#1\dimexpr\macc@kerna}
\newcommand*\widebar[1]{\@ifnextchar^{{\wide@bar{#1}{0}}}{\wide@bar{#1}{1}}}
\newcommand*\wide@bar[2]{\if@single{#1}{\wide@bar@{#1}{#2}{1}}{\wide@bar@{#1}{#2}{2}}}
\newcommand*\wide@bar@[3]{%
  \begingroup
  \def\mathaccent##1##2{%
    \let\mathaccent\save@mathaccent
    \if#32 \let\macc@nucleus\first@char \fi
    \setbox\z@\hbox{$\macc@style{\macc@nucleus}_{}$}%
    \setbox\tw@\hbox{$\macc@style{\macc@nucleus}{}_{}$}%
    \dimen@\wd\tw@
    \advance\dimen@-\wd\z@
    \divide\dimen@ 3
    \@tempdima\wd\tw@
    \advance\@tempdima-\scriptspace
    \divide\@tempdima 10
    \advance\dimen@-\@tempdima
    \ifdim\dimen@>\z@ \dimen@0pt\fi
    \rel@kern{0.6}\kern-\dimen@
    \if#31
      \overline{\rel@kern{-0.6}\kern\dimen@\macc@nucleus\rel@kern{0.4}\kern\dimen@}%
      \advance\dimen@0.4\dimexpr\macc@kerna
      \let\final@kern#2%
      \ifdim\dimen@<\z@ \let\final@kern1\fi
      \if\final@kern1 \kern-\dimen@\fi
    \else
      \overline{\rel@kern{-0.6}\kern\dimen@#1}%
    \fi
  }%
  \macc@depth\@ne
  \let\math@bgroup\@empty \let\math@egroup\macc@set@skewchar
  \mathsurround\z@ \frozen@everymath{\mathgroup\macc@group\relax}%
  \macc@set@skewchar\relax
  \let\mathaccentV\macc@nested@a
  \if#31
    \macc@nested@a\relax111{#1}%
  \else
    \def\gobble@till@marker##1\endmarker{}%
    \futurelet\first@char\gobble@till@marker#1\endmarker
    \ifcat\noexpand\first@char A\else
      \def\first@char{}%
    \fi
    \macc@nested@a\relax111{\first@char}%
    \fi
  \endgroup
}
\makeatother

\theoremstyle{plain}
\newtheorem{thm}{Theorem}
\theoremstyle{plain}
\newtheorem{prop}[thm]{Proposition}
\theoremstyle{plain}

\theoremstyle{plain}
\newtheorem{rmk}[thm]{Remark}
\theoremstyle{plain}

\theoremstyle{plain}
\newtheorem{lemm}[thm]{Lemma}
\theoremstyle{plain}
\newtheorem{hypo}{Assumption}
\theoremstyle{plain}
\newtheorem{expl}{Example}
\DeclareMathOperator{\e}{e}

\newcommand{\varx}{x}
\newcommand{\vary}{y}
\newcommand{\vart}{t}
\newcommand{\fun}{u}
\newcommand{\Alin}{A_\mathrm{L}}

\newcommand{\Cdiff}{\mathfrak{C}}
\newcommand{\CFP}{\mathscr{C}}
\newcommand{\Ccluster}{\widehat{C}}

\renewcommand{\le}{\leqslant}
\renewcommand{\ge}{\geqslant}
\renewcommand{\leq}{\leqslant}
\renewcommand{\geq}{\geqslant}

\begin{document}

\title{A mathematical justification of the finite time approximation of Becker--D\"oring equations by a Fokker--Planck dynamics}

\author{Gabriel Stoltz$^1$ and Pierre Terrier$^{1,2}$\\
\small $^1$ Université Paris-Est, CERMICS (ENPC), Inria, F-77455 Marne-la-Vallée, France \\
\small $^2$ CEA, DEN, Service de Recherches de Métallurgie Physique, UPSay, F-91191 Gif-sur-Yvette, France \\
}

\date{\today}

\maketitle


\abstract{
  The Becker--D\"oring equations are an infinite dimensional system of ordinary differntial equations describing coagulation/fragmentation processes of species of integer sizes. Formal Taylor expansions motivate that its solution should be well described by a partial differential equation for large sizes, of advection-diffusion type, called Fokker--Planck equation. We rigorously prove the link between these two descriptions for evolutions on finite times rather than in some hydrodynamic limit, motivated by the results of numerical simulations and the construction of dedicated algorithms based on splitting strategies. In fact, the Becker--D\"oring equations and the Fokker--Planck equation are related through some pure diffusion with unbounded diffusion coefficient. The crucial point in the analysis is to obtain decay estimates for the solution of this pure diffusion and its derivates to control remainders in the Taylor expansions. The small parameter in this analysis is the inverse of the minimal size of the species.
}

\section{Introduction}

Simulating the ageing of materials over a long period of time remains a challenge in the materials science community. Purely atomistic approaches, such as molecular dynamics or kinetic Monte Carlo~\cite{bortz1975new,ising1925beitrag,soisson2010atomistic,voter2007introduction,young1966monte} do not allow to reach times as long as years of ageing. To achieve this goal, mean-field models have been developed. One model, called Cluster Dynamics, has been considered in the community of nuclear materials~\cite{ghoniem1980numerical,barbu2007cluster,jourdan2014efficient} in order to study the evolution of defects under irradiation. In fact, this model coincides with the celebrated Becker--D\"oring (BD) equations, first proposed in~\cite{BD35} and then modified in~\cite{PL79}, which allow to simulate coagulation/fragmentation processes in various fields, including biology. It consists in simulating the evolution of concentration of various species, here clusters of defects such as vacancies or other self defects, solute gas, etc. From a mathematical viewpoint, BD is an infinite set of ordinary differential equations (ODEs), one for each type of defect. We focus in this work on a simple but paradigmatic example of a single species coagulation/fragmentation process, which corresponds for instance to vacancy clustering in materials science.

Let us first recall the BD equations, denoting by $C_n$ the concentration of clusters composed of $n$ vacancies. The evolution of $C_n$ is given by 
\begin{equation}
\label{eq:EDO-Vacancy}
\frac{dC_n}{dt} = \beta_{n-1}C_{n-1}C_1 - (\beta_n C_1 + \alpha_n) C_n + \alpha_{n+1}C_{n+1},
\end{equation}
where $\alpha_n$ and $\beta_n$ are respectively called emission and absorption coefficients, while $C_1$ represents the concentration of clusters composed of only one vacancy. Equation~\eqref{eq:EDO-Vacancy} describes a simple process where clusters of defects can either emit or absorb a single vacancy. More precisely, the term $\beta_{n-1} C_{n-1} C_1$ describes the increase in the population of clusters of size $n$ coming from clusters of size $n-1$ absorbing a single vacancy, while the term $\alpha_{n+1} C_{n+1}$ describes the increase in the population of clusters of size $n$ coming from clusters of size $n+1$ emitting a single vacancy. Finally, the term $ - (\beta_n C_1 + \alpha_n) C_n $ encodes the rate of decrease of clusters of size $n$ arising from their transformation into clusters of sizes $n-1$ or $n+1$. Single vacancies are considered as mobile clusters and their evolution is related to the evolution of all other clusters as follows:
\begin{equation}
\label{eq:EDO-Mono}
\frac{dC_1}{dt} = -2\beta_1 C_1^2 - \sum_{n\ge 2} \beta_n C_n C_1 + \sum_{n\ge 2} \alpha_n C_n + \alpha_2 C_2.
\end{equation}
The latter equation is determined by the requirement that the total quantity of matter is conserved, namely
\begin{equation}
\label{eq:conservation-law}
\frac{d}{dt}{\left(C_1 + \sum_{n\ge 2} n C_n \right)} = 0.
\end{equation}
The reasons such a model is a simplification of complex phenomena occurring in real materials are twofold. First, mobile clusters can be of size greater than one. Equation~\eqref{eq:EDO-Vacancy} can be enriched with terms describing the absorption or emission of clusters of sizes $m \ge 1$, with equations similar to~\eqref{eq:EDO-Mono} describing the evolution of concentrations of sizes $m\ge 1$. Second, clusters can be made of different types of defects, \textit{e.g.} vacancies and helium atoms in iron. Therefore, defect concentrations are in general indexed by $k$-tuples, where $k$ is the number of types of defects. 

Various properties of the BD equations~\eqref{eq:EDO-Mono}-\eqref{eq:conservation-law} are reviewed in~\cite{Slemrod00,HY17}. The study of the well-posedness of the BD equations was initiated in~\cite{ball1986becker}, with uniqueness results refined in~\cite{Laurencot2002becker}. Many works, starting with~\cite{ball1986becker}, then adressed the longtime behavior of BD, which depends on whether the total mass is below some treshold value, in which case precise rates of convergence to a steady-state solution with fixed mass can be obtained (see~\cite{JN03} and subsequent works as reviewed in~\cite{HY17}); or above the treshold value, in which case some mass is lost in the steady-state~\cite{Slemrod89}, a signature of some intriguing phase transition. Another viewpoint on the longtime behavior of BD is to obtain some average description of the dynamics under the form of a partial differential equation (PDE) under a suitable space-time scaling corresponding to some hydrodynamic limit, as initiated in~\cite{Penrose97} and pursued in many works, including~\cite{Laurencot2002becker,CGPV02,niethammer03,niethammer04}. The limiting PDE is a nonlinear transport equation, the so-called Lifshitz--Slyozov equation, or a variant of it. This PDE describes the longtime behavior of large clusters, and is therefore a model relevant for the coalescence regime.

Our concern here is rather on a non-asymptotic phase of the process, where nucleation and growth are still important, and where the concentration of small defects remains high because of the irradiation. This regime is not well described by the Lifshitz--Slyozov equation, but rather by Fokker--Planck type equations, which can be seen as nonlinear transport equations supplemented by a diffusion term. Fokker--Planck equations related to BD were first presented in~\cite{goodrich1964nucleation} and are still used and developed in more recent works in the materials science community~\cite{jourdan2014efficient,jourdan2016accurate,terrier2017coupling}. They have also been considered in mathematical studies, where the diffusion is obtained as a higher order correction term in scaling limits~\cite{Velazquez98,HC99,CGPV02}. Note however that the diffusion part typically vanishes in the longtime/large scale limit in these models. 

Our motivation for studying Fokker--Planck type approximations at finite times rather than large times comes from numerical considerations. A numerical approach to solve BD is to fix the maximal size $N_\mathrm{max}$ of the clusters in the simulation, which amounts to solving a system of $N_\mathrm{max}$ ODEs. In practice, the ODE system is stiff so that dedicated solvers are required, see~\cite{CDW95}, as well as~\cite{jourdan2014efficient} for the applications we have in mind. However, it can become computationally impossible to solve the ODEs. For example, a system with several types of defects such as vacancies (V) and helium atoms (He), might contain up to~$N_\mathrm{max} = N_\mathrm{max,V}  N_\mathrm{max,He} = 10^{6} \times 10^5$ equations. This motivated the development of various approximations, for instance based on some coarse-graining procedure~\cite{CDW95,kiritani1973analysis,golubov2001grouping}. An alternative strategy consists in coupling the EDO system with a macroscopic description of the materials in terms of a PDE, formally obtained by a second order discretization based on the assumption that $C_n(t) \simeq \mathscr{C}(t,n)$ for some smooth function~$\mathscr{C}$ (we recall the heuristic derivation of this equation in Section~\ref{sec:standard-approx}):
\begin{equation}
\label{eq:Fokker--Planck}
\frac{\partial \CFP}{\partial t} = - \frac{\partial (F\CFP)}{\partial x} + \frac{1}{2}\frac{\partial^2 (D\CFP)}{\partial x^2}.
\end{equation}
The Fokker--Planck equation~\eqref{eq:Fokker--Planck} is characterized by the drift $F$ and the diffusion $D$, both coefficients depending on the coefficients of absorption and emission $\beta$ and $\alpha$.

\begin{figure}
\centering
\includegraphics[width=0.5\textwidth]{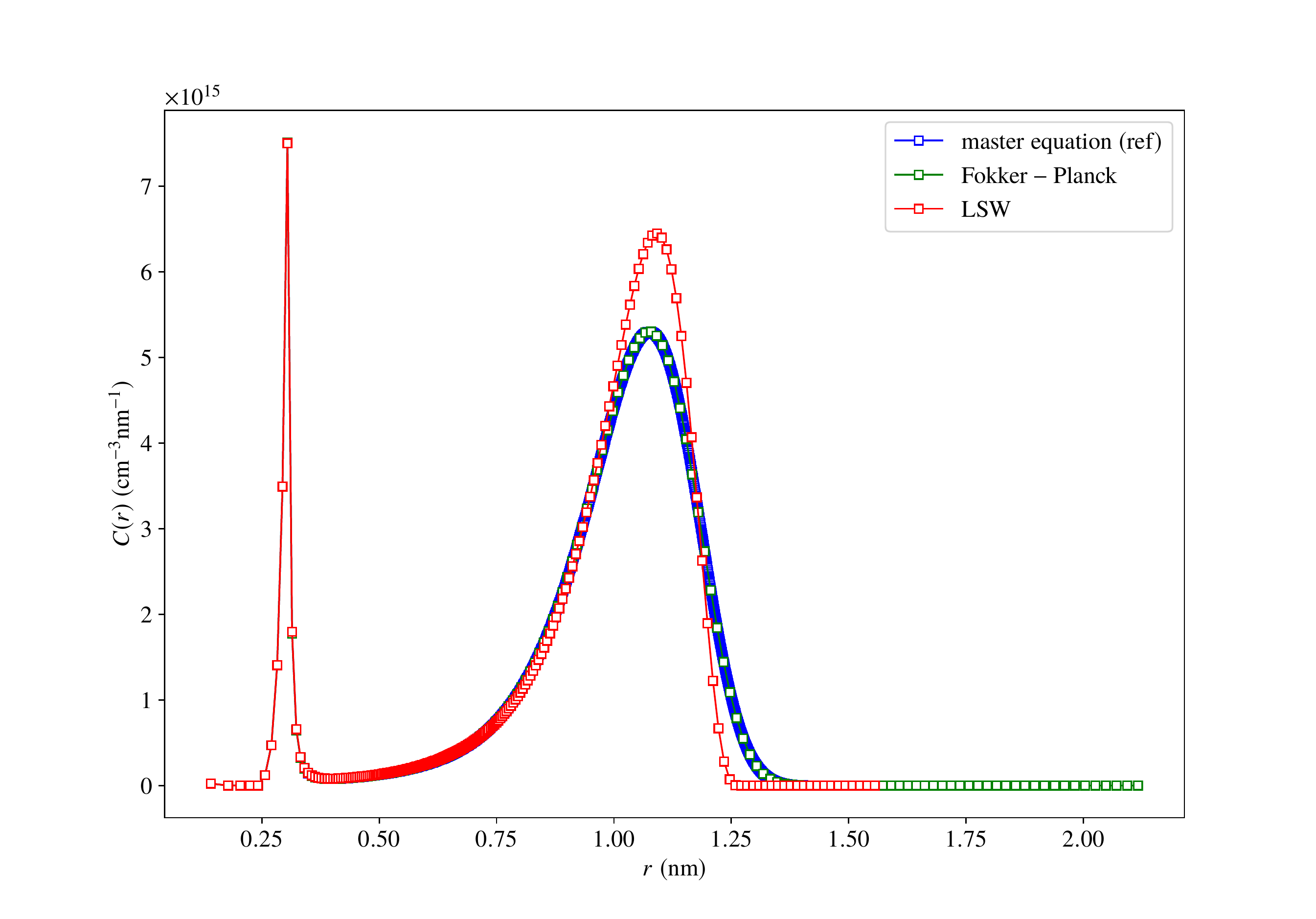}
\caption{Comparison of the growth of vacancy clusters under irradiation, at time $t = 10^5$ s, of the solutions of BD equations \eqref{eq:EDO-Vacancy}-\eqref{eq:EDO-Mono} ('master equation'), the Fokker--Planck dynamics~\eqref{eq:Fokker--Planck} and the Lifshitz-Slyozov like dynamics LSW (\emph{i.e.} \eqref{eq:Fokker--Planck} with $D=0$). See~\cite{Terrier} for details on the simulations.
}
\label{fig:comp-LSW-FP}
\end{figure}

While the approximation~\eqref{eq:Fokker--Planck} gives accurate results in practice at finite times for large cluster sizes (when compared to the solution of the full ODE system), the consistency of this approach has never been rigorously proven, and the approximation error never been quantified in function of the minimal size of the clusters for which it is used. Let us emphasize that, as demonstrated in Figure~\ref{fig:comp-LSW-FP}, it is crucial to include the diffusion term in the PDE approximating BD at finite times, as highlighted by the fact the solution of~\eqref{eq:Fokker--Planck} with $D=0$ does not capture correctly the solution of the ODE system. The numerical results reported in this picture also show that the Fokker--Planck approximation can be valid for very small cluster sizes.

A difficulty in making the heuristic argument of~\cite{goodrich1964nucleation} leading to~\eqref{eq:Fokker--Planck} rigorous is that the approximation on which the derivation relies is based on a Taylor expansion of order 2 for a mesh with fixed spacing 1. This is in contrast with the usual approximations in the mathematics literature which rely on some space-time scaling naturally leading to some continuum limit as in~\cite{Velazquez98,HC99,CGPV02}. One would need decay estimates on the derivatives of the solution to justify the Taylor expansion used in~\cite{goodrich1964nucleation}. While the recent work~\cite{CS19} considered the well-posedness and regularity of Fokker--Planck type dynamics~\eqref{eq:Fokker--Planck} related to BD, as well as their convergence to equilibrium, we are not aware of results on decay estimates for derivatives of solutions of~\eqref{eq:Fokker--Planck}. In fact, we do not directly work on this equation, but perform a change of variables $q = Q(t,x)$ based on the characteristics of the transport part of the equation, which leads to another PDE, a pure diffusion of the form
\begin{equation}
\label{eq:diffusion-equation-intro}
\frac{\partial \Cdiff}{\partial t} = \frac{1}{2}\sigma^2(q) \frac{\partial^2\Cdiff}{\partial q^2},
\end{equation}
where the diffusion coefficient $\sigma^2$ depends on the coefficients $\alpha$ and $\beta$. The interest of the diffusion equation is that it is now possible to make precise the decay of the derivatives of $\mathfrak{C}$. The PDE~\eqref{eq:diffusion-equation-intro} allows us to relate BD and its Fokker--Planck approximation. The small parameter in this analysis is the inverse of the minimal cluster size~$M$. Due to the fact that BD is inherently nonlinear, we also introduce a splitting of the dynamics in order to restrict the nonlinearity to the evolution~\eqref{eq:EDO-Mono} of single vacancies. This splitting allows us to work in a simpler framework and to prove rigorously the link between the Fokker--Planck approximation and BD for larger cluster sizes. Moreover, this splitting is also of interest for numerical simulations~\cite{terrier2017coupling}.

%

\medskip

This article is organized as follows. We first present results concerning BD in Section~\ref{sec:well-posedness}. We start by recalling well-posedness results in Section~\ref{sec:-well-posedness-solution}, and next discuss in Section~\ref{sec:splitting} the convergence of a splitted dynamics where~\eqref{eq:EDO-Vacancy} and~\eqref{eq:EDO-Mono} are integrated successively. The proofs of the results given in Section~\ref{sec:well-posedness} are postponed to Appendices~\ref{sec:appendix-well-posedness} and~\ref{sec:appendix-splitting}. We then focus our attention on the Fokker--Planck approximation in Section~\ref{sec:Fokker--Planck}. After a heuristic derivation of the Fokker--Planck approximation, we present the diffusion equation~\eqref{eq:diffusion-equation-intro}, which is a key equation in this work. We state decay results on its solutions, which allows us to relate the diffusion equation and both the Fokker--Planck equation~\eqref{eq:Fokker--Planck} and BD~\eqref{eq:EDO-Vacancy} for large cluster sizes. The proofs of the technical results of Section~\ref{sec:Fokker--Planck} are gathered in Appendix~\ref{sec:appendix-proofs-FPCD}.

\section{Results on the Becker--D\"oring equations and their splitting}
\label{sec:well-posedness}

We recall in Section~\ref{sec:-well-posedness-solution} the mathematical framework in which BD is well posed~\cite{ball1986becker,Laurencot2002becker}, and introduce there some notation which will prove useful in the remainder of the section. We next introduce in Section~\ref{sec:splitting} a splitting of the dynamics and prove that it is consistent of order 1. We denote in this section the solution of BD by $u=(u_1,u_2,\dots)$.

\subsection{Well-posedness of the Becker--D\"oring equations}
\label{sec:-well-posedness-solution}

We consider the full BD (\textit{i.e.} equations~\eqref{eq:EDO-Vacancy}--\eqref{eq:EDO-Mono}), which is a nonlinear dynamics. We recall here the existence results by Ball, Carr and Penrose~\cite{ball1986becker}, and the refined existence result obtained in~\cite{Laurencot2002becker}. We also give an alternative proof of the uniqueness in Appendix~\ref{sec:appendix-well-posedness}, based on the dissipativity of the underlying operator. We work on the Hilbert space $\mathcal{H} = \ell^2(\mathbb{N}^*,\mathbb{R})$, endowed with its natural norm $\|\cdot\|$ and inner product $\langle \cdot, \cdot \rangle$. Unless stated otherwise, the norm $\|\cdot\|$ is the natural norm of $\mathcal{H}$ or the norm of bounded operators on $\mathcal{H}$, depending on the context.

Consider $u= (u_1,u_2,\cdots,u_n,\cdots) \in \mathcal{H}$ and denote by $(e_n)_{n \in \mathbb{N}^*}$ the orthonormal basis of $\mathcal{H}$ defined by $(e_n)_i = \delta_{ni}$, where $\delta_{ni}$ is the usual Kronecker symbol. In particular, $\langle u,\e_i\rangle = u_i$. BD can be written with this notation as the following Cauchy problem in $\mathcal{H}$:
\begin{equation}
\label{problem:P3}
\tag{BD}
{\left\{
\begin{aligned}
&\frac{du}{dt} = A(u_1)u,\\
&u(0) = u^0,
\end{aligned}
\right.}
\end{equation}
where the quasi-linear operator $A$ is defined, for all $v \in \mathcal{H}$, by:
\begin{equation}
\label{eq:operator-A}
\begin{aligned}
& A(v_1)e_1 = -2\beta_1 v_1 e_1 + \beta_1 v_1 e_2,\\
& A(v_1)e_n = (\alpha_n -\beta_n v_1)e_1 + \alpha_n e_{n-1} -(\beta_n v_1 + \alpha_n) e_n + \beta_n v_1 e_{n+1},\qquad n\ge 2.\\
\end{aligned}
\end{equation}
Alternatively, $A$ can be written as the following infinite matrix:
\begin{equation}
A(v_1) = \begin{pmatrix}
-2\beta_1 v_1 & 2\alpha_2-\beta_2 v_1 & \alpha_3 - \beta_3 v_1 & \alpha_4 -\beta_4 v_1 & \cdots \\
\beta_1 v_1 & -(\beta_2 v_1 +\alpha_2) & \alpha_3 & 0 & \cdots \\
0 & \beta_2 v_1 & -(\beta_3 v_1 +\alpha_3) & \alpha_4 & \cdots \\ 
0 & 0 & \beta_3 v_1 & -(\beta_4 v_1 +\alpha_4) & \ddots \\ 
\vdots & \vdots & \vdots & \ddots & \ddots &
\end{pmatrix} .
\end{equation}
The main difficulty of the problem~\eqref{problem:P3} comes from the unboundedness of the coefficients $\alpha_n$ and $\beta_n$. As will be made clear below, it is nonetheless possible to obtain existence and uniqueness results for coefficients which do not grow too fast. In fact, physically, $\alpha_n, \beta_n = \mathrm{O}(n^\gamma)$, with $\gamma = 1/3$ for vacancies and solutes which generate three-dimensional objects such as bubbles, and $\gamma = 1/2$ for interstitials which generate two-dimensional objects such as loops (see Example~\ref{expl:alpha-beta-discret} below). Nevertheless, in order to prove the existence and uniqueness of smoooth solutions, the following assumptions on $\alpha$ and $\beta$ are sufficient.

\begin{hypo}
\label{hypo:alpha-beta-discrete}
The sequences $\alpha = (\alpha_n)_{n\ge 1}$ and $\beta = (\beta_n)_{n\ge 1}$ are two sequences of non-negative real numbers, and there exists $B \in \mathbb{R}_+$ such that
\begin{equation}
\forall\, n\ge 1, \qquad \left|\alpha_{n+1}-\alpha_n\right| \leq B, \qquad \left|\beta_{n+1}-\beta_n\right| \leq B.
\label{H1}
\tag{H1}
\end{equation}
\end{hypo}
In fact, the proof of the uniqueness only requires $\alpha_{n+1}-\alpha_n \le B$ and $\beta_{n+1}-\beta_n \ge -B$ as in~\cite[Theorem~2.1]{Laurencot2002becker} (see Appendix~\ref{sec:appendix-well-posedness}). We consider the stronger conditions~\eqref{H1} in order to apply the regularity results from~\cite[Theorem~3.2]{ball1986becker}. Let us now present a specific example, which we will use throughout this work to illustrate the relevance of our assumptions.

\begin{expl}
\label{expl:alpha-beta-discret}
In many physical models~\cite{ghoniem1980numerical,ovcharenko2003gmic++,barbu2007cluster,jourdan2014efficient}, the expression of $\alpha_n$ and $\beta_n$ are chosen as follows:
\begin{equation}
\label{eq:beta-example}
\beta_n = {\left(\frac{48\pi^2}{V_\mathrm{at}^2}\right)}^{\gamma} D r_n,
\qquad 
\alpha_n = {\left(\frac{48\pi^2}{V_\mathrm{at}^2}\right)}^{\gamma} D r_n \exp{\left(-\frac{E_\mathrm{v}^\mathrm{f}}{k_\mathrm{B}T}\right)}\exp{\left(\frac{\omega}{r_n}\right)},
\end{equation}
where $D$ is the diffusion coefficient of mobile clusters, $V_\mathrm{at}$ the atomic volume, $E_\mathrm{v}^\mathrm{f}$ the formation energy of a vacancy, $k_\mathrm{B}$ the Boltzmann constant, $T$ the temperature, $\omega$ a parameter related to the type of clusters (vacancies or interstitials) and $r_n = n^\gamma$, with $\gamma \in \{\frac{1}{3}, \frac{1}{2}\}$. It is easy to see that the sequences $\alpha$ and $\beta$ indeed satisfy Assumption~\ref{hypo:alpha-beta-discrete}.
\end{expl}

In order to prove the global-in-time well-posedness, we introduce the following subset of $\mathcal{H}$ (already considered in~\cite{ball1986becker}):
\begin{equation}
\label{eq:def-Q}
\mathcal{Q} = {\left\{ u\in \mathcal{H}\ \middle|\ \forall n \geq 1, \ u_n \ge 0, \ \ \sum_{n\ge 1} nu_n < +\infty \right\}}.
\end{equation}
The condition $\sum n u_n < +\infty$ translates the physical fact that the total quantity of matter is finite. In fact, this quantity is conserved by the BD dynamics~\eqref{eq:EDO-Vacancy}--\eqref{eq:EDO-Mono}. For any element $u \in \mathcal{Q}$, we define
\begin{equation}
\label{eq:def-Q0}
Q(u) = \sum_{n\ge 1} n u_n.
\end{equation}
\begin{rmk}
\label{rmk:def-D-pos}
In view of~\eqref{eq:def-Q} and Assumption~\ref{hypo:alpha-beta-discrete}, an element $u \in \mathcal{Q}$ satisfies in particular 
\begin{equation}
\label{eq:in-ell-1}
0\le \sum_{n\ge 1} \alpha_n u_n < +\infty \qquad \text{and} \qquad 0\le \sum_{n\ge 1} \beta_n u_n  < +\infty.
\end{equation}
This means that the sequences $(\alpha_n u_n)_{n\ge 1}$ and $(\beta_n u_n)_{n\ge 1}$ are in $\ell^1(\mathbb{N}^*,\mathbb{R})$ when $u\in \mathcal{Q}$.
\end{rmk}

We are now in position to recall the following result on the well-posedness of BD.

\begin{thm}[see~\cite{ball1986becker,Laurencot2002becker}]
\label{thm:long-time}
Fix an initial condition $u^0 \in \mathcal{Q}$ and suppose that Assumption~\ref{hypo:alpha-beta-discrete} holds true. Then, there exists a unique global-in-time classical solution $u \in \mathcal{C}^0(\mathbb{R}_+,\mathcal{Q})\cap \mathcal{C}^1(\mathbb{R}_+,\mathcal{H})$ of the problem~\eqref{problem:P3}. Moreover,
\begin{equation}
\forall\, t\ge 0,\qquad Q(u(t)) = Q{\left(u^0\right)} \qquad\text{and}\qquad {\left\|u(t)\right\|} \le \frac{\pi Q{\left(u^0\right)}}{\sqrt{6}}.
\end{equation}
\end{thm}

The last estimate is easily obtained from the bound on~$Q(u(t))$. Let us mention that an alternative proof for the existence of the solution to BD is proposed in~\cite{Terrier}. Instead of approximating the infinite dimensional BD equations by a finite dimensional system as in~\cite{ball1986becker}, we work with an infinite dimensional dynamics on~$\mathcal{H}$ where the coefficients $\alpha_n,\beta_n$ are approximated by bounded coefficients $\alpha_n/(1+\varepsilon n^p)$ and $\beta_n/(1+\varepsilon n^p)$ for $\varepsilon>0$ and $p\ge 2$. The scheme of the proof of the existence remains however the same as in~\cite{ball1986becker}: it is first shown that the regularized dynamics is well posed in~$\mathcal{H}$ and preserves the total mass; a solution is then obtained by a compactness argument. The uniqueness part however follows a strategy different from the one in~\cite{Laurencot2002becker}, based on the dissipativity of the operator $A(b)$ for~$b \ge 0$ fixed. This dissipativity property is in fact useful for the numerical analysis provided in Section~\ref{sec:splitting}, which is why we present these estimates in Appendix~\ref{sec:appendix-well-posedness}, where our alternative proof for the uniqueness of the dynamics can also be read.

\subsection{Splitting of the dynamics and qualitative properties}
\label{sec:splitting}

We discuss in this section some properties of the dynamics obtained by splitting the nonlinear dynamics~\eqref{problem:P3} into two sub-dynamics, one on the first concentration only and another one on the remaining concentrations. Let us emphasize that this corresponds to a somewhat ideal splitting, where only the error with respect to the integration time of each subdynamics is taken into account. In particular, no error related to the use of integration schemes for the dynamics on the remaining concentrations is considered (as errors arising from the truncature of the size of the system, or the use of time-stepping methods).

The motivations for considering the properties of this ideal splitting are twofold.
First, it allows to restrict the nonlinearity to one equation, while the remainder of the dynamics becomes linear. It is one of the key features we used in~\cite{terrier2017coupling} for an efficient numerical integration of cluster dynamics.
Second, the validity of the Fokker--Planck approximation is obtained only for linear dynamics. The proof of such an approximation in Section~\ref{sec:Fokker--Planck} is performed for the linear sub-dynamics of clusters of larger sizes.
Note that the splitting introduced in this section can be generalized to a splitting on a first dynamics on small clusters from sizes $1$ to $M$ (which can then be integrated with any time-stepping method), and on a second dynamics on larger clusters of sizes greater than $M+1$ for some $M\ge 1$. The estimates we use in our proof, which rely on an explicit integration of the dynamics of the first concentration (see Lemma~\ref{lemm:estimate-for-varphi}) should then be generalized by more abstract results. Although this is possible, we refrain from doing so in order to keep the presentation more readable.

The splitting we consider here is a simple Lie-Trotter splitting~\cite{trotter1959product}, although more elaborate splitting such as Strang splittings can of course be considered. We prove that the associated dynamics is consistent with the full dynamics~\eqref{problem:P3} in the natural norm of~$\mathcal{H}$. Although this result is of course expected, proving it rigorously requires controlling the behavior of the solution of the splitted dynamics, for which no mass conservation holds. We need to strengthen Assumption~\ref{hypo:alpha-beta-discrete} to this end, in order to control the first and second derivatives of the solution.

\begin{hypo}
\label{hypo:alpha-beta-gamma}
There exist $0 \le \gamma \le 1/2$ and $K \in \mathbb{R}_+^*$ such that $0\le \alpha_n, \beta_n \le K n^\gamma$ for all $n \ge 1$.
\end{hypo}

Note that Assumption~\ref{hypo:alpha-beta-gamma} clearly holds true for Example~\ref{expl:alpha-beta-discret}. Let us now write more explicitly the splitted dynamics we consider. The sub-dynamics for the first concentration only reads 
\begin{equation}
\label{eq:sub-dynamic1}
{\left\{
    \begin{aligned}
    &\frac{du_1}{dt} = -2\beta_1 u_1^2 - {\left(\sum_{n\ge2} \beta_n u_n \right)}u_1 + {\left(\sum_{n\ge2} \alpha_n u_n + \alpha_2 u_2 \right)},\\
    &\frac{du_n}{dt} = 0,\qquad \forall n\ge 2,
    \end{aligned}
\right.}
\end{equation}
\textit{i.e.} $u_n$ is fixed for $n\ge 2$. We denote by $\varphi_t^{(u_2,\cdots)}$ the flow of this dynamics, or simply $\varphi_t$ when the dependence is clear. Note that this sub-dynamics is well-posed (see~\eqref{eq:sol-u1S} in Appendix~\ref{sec:appendix-splitting-proof}). The second sub-dynamics, for the remaining concentrations, reads
\begin{equation}
\label{eq:sub-dynamic2}
{\left\{
    \begin{aligned}
    &\frac{du_1}{dt} = 0,\\
    &\frac{du_n}{dt} = \beta_{n-1}u_1 u_{n-1} - {\left(\beta_n u_1 + \alpha_n \right)}u_n + \alpha_{n+1}u_{n+1},\qquad \forall n\ge 2,
    \end{aligned}
\right.}
\end{equation}
\textit{i.e.} $u_1$ is fixed. We denote by $\chi_t^{u_1}$ the flow of this dynamics, or simply $\chi_t$ when the dependence is clear. This sub-dynamics is also well-posed (see Proposition~\ref{prop:splitting-semigroup} in Appendix~\ref{sec:appendix-sub2}). One step of the splitted dynamics is encoded by the mapping $u^{k+1} = S_{\Delta t}{\left(u^k\right)}$ for a given time step $\Delta t > 0$, defined as
\begin{enumerate}
\item update the first concentration as $u^{k+1}_1 = \varphi_{\Delta t}^{{\left(u^k_2,\cdots\right)}}(u^k_1)$,
\item update the remaining concentrations as ${\left(u^{k+1}_2,\cdots\right)} = \chi_{\Delta t}^{u^{k+1}_1}{\left(u^k_2,\cdots\right)}$.
\end{enumerate}
The iterates defined as $u^{k} = S_{\Delta t}(u^{k-1})$ for $k\ge 1$ and some initial condition $u^0$ are an approximation of $u(k\Delta t)$, the solution of~\eqref{problem:P3} with initial condition $u^0$ at time $k\Delta t$. The following proposition states that the Lie-Trotter splitting is consistent of order 1 in the norm of~$\mathcal{H}$ (see Section~\ref{sec:appendix-splitting-proof} for the proof).
\begin{prop}
\label{prop:splitting}
Fix an initial condition $u^0 \in \mathcal{Q}$ with $u_1^0 > 0$ and a time $\tau > 0$. Suppose that Assumptions~\ref{hypo:alpha-beta-discrete} and~\ref{hypo:alpha-beta-gamma} hold true. Suppose also that there exists $k\ge 2$ such that $u_k^0 > 0$. Then, there exist a constant $L{\left(\tau, u^0\right)} \in \mathbb{R}_+$ and a time step $\Delta t_* > 0$ such that
\begin{equation}
\forall\, \Delta t \in (0,\Delta t_*],\qquad \forall\, 0\le n \le \frac{\tau}{\Delta t}, \qquad {\left\|u(n\Delta t) - u^n\right\|} \le L{\left(\tau,u^0\right)} \Delta t.
\end{equation}
\end{prop}
The assumption that the initial condition is non-zero in the sub-domain $[2,+\infty)$ ensures that the subdynamics~\eqref{eq:sub-dynamic1} are well posed (see Lemma~\ref{lemm:estimate-for-varphi}). In order to prove Proposition~\ref{prop:splitting}, we need some control over the first and second derivatives of the solution of~\eqref{problem:P3} (see Section~\ref{sec:appendix-splitting-uprime}).

\begin{prop}
\label{prop:uprime}
Suppose that Assumptions~\ref{hypo:alpha-beta-discrete} and~\ref{hypo:alpha-beta-gamma} hold true. Let $u$ be the classical solution of problem~\eqref{problem:P3} with initial condition $u^0 \in \mathcal{Q}$ and $u_1^0 > 0$. Fix a time $\tau > 0$, a constant $Q_* \in\mathbb{R}_+$ and suppose that $Q{\left(u^0\right)} \le Q_*$. Then, $u \in \mathcal{C}^2([0,\tau],\mathcal{H})$ and there exists $R(Q_*) \in \mathbb{R}_+$ such that, for all $0\le t\le \tau$,
\begin{equation}
{\left\| \frac{du}{dt}(t)\right\|}, {\left\| \frac{d^2u}{dt^2}(t)\right\|} \le R(Q_*).
\end{equation} 
\end{prop}

In the following section, we will prove that the linear problem~\eqref{eq:sub-dynamic2} is related, up to a small error, to a diffusion equation, which is the key equation to relate the Fokker--Planck approximation and the Becker--D\"oring equations. 

\section{The Fokker--Planck approximation in the linear case}
\label{sec:Fokker--Planck}

The Fokker--Planck approximation~\eqref{eq:Fokker--Planck} is widely used in the materials science community to approximate the dynamics of large size clusters~\cite{goodrich1964nucleation,wolfer1977theory}. This approximation gives accurate results in very good agreement with BD when sufficiently precise numerical schemes are used for the simulation~\cite{jourdan2016accurate}. It proves to be efficient and speeds up the simulations for complex systems~\cite{jourdan2014efficient}. We also report a very good agreement between the solution of the exact BD and a coupling approach solving the ODEs for small size clusters and the Fokker--Planck PDE for large size ones~\cite{terrier2017coupling}. Nevertheless, the agreement between BD and its Fokker--Planck approximation at finite times has never been quantified to our knowledge. The only results we are aware of concerning the relationship between BD and FP are based on hydrodynamic limits or space/time rescalings~\cite{Velazquez98,HC99,CGPV02}, which lead to a vanishing diffusion. We provide in this section a proof of the correctness of the Fokker--Planck limit using stochastic techniques, and quantify the approximation error. The small parameter in this analysis is the inverse of the minimal cluster size~$M$.

In the whole section, the concentration $C_1$ of single vacancies is supposed to be fixed. The section is organized as follows. We first present a formal derivation of the Fokker--Planck approximation in Section~\ref{sec:standard-approx}. We next heuristically derive a reformulation of the Fokker--Planck approximation in the form of a diffusion equation (see Section~\ref{sec:formal-results}), for which we state a key result on the decay of the spatial derivatives of the solution in Section~\ref{sec:thm-decay}. This result allows us to rigorously establish the link between the Fokker--Planck approximation and BD equations, and to quantify errors as a function of the minimal cluster size (see Section~\ref{sec:Diff2CDandFP}).

\subsection{Heuristic derivation of the Fokker--Planck approximation}
\label{sec:standard-approx}

We describe here the derivation of the Fokker--Planck approximation as originally presented in the materials science community~\cite{goodrich1964nucleation}, pointing out the parts of the argument which require a more rigorous mathematical analysis. Let us emphasize that all computations presented here are formal.

Define the regular mesh $(x_n)_{n\ge 0}$ of $[0,+\infty)$ by $x_n = n$. The mesh size is $\Delta x = 1$. Let us assume that there exist smooth functions $\alpha,\beta \in \mathcal{C}^\infty(\mathbb{R}_+, \mathbb{R}_+)$ and $\CFP \in \mathcal{C}^\infty(\mathbb{R}_+ \times \mathbb{R}_+, \mathbb{R}_+)$ such that, for all $n \in \mathbb{N}^*$ and $t\ge 0$,
\begin{equation}
\CFP(t,x_n) = C_n(t), \qquad \alpha(x_n) = \alpha_n, \qquad \beta(x_n) = \beta_n.
\end{equation}
When $C_1$ is fixed, Equation~\eqref{eq:EDO-Vacancy} can be written, for $n\ge 2$, as
\begin{equation}
\begin{aligned}
\frac{\partial \CFP}{\partial t}(t,x_n) =&\ C_1{\left[\beta(x_{n-1})\CFP(t,x_{n-1}) - \beta(x_{n})\CFP(t,x_n)\right]} + \alpha(x_{n+1})\CFP(t,x_{n+1}) - \alpha(x_{n})\CFP(t,x_n).
\end{aligned}
\label{eq:EDO-continuous}
\end{equation}
Let us emphasize that fixing $C_1$ is crucial for the argument. In practice, this arises through the splitting described in Section~\ref{sec:splitting}. By Taylor expansions at order 2,
\[
  \begin{aligned}
    \alpha(x_{n+ 1})\CFP(t,x_{n + 1}) & \simeq \alpha(x_{n})\CFP(t,x_n) + \Delta x \frac{\partial}{\partial x}(\alpha \CFP)(t,x_n)  + \frac{1}{2}(\Delta x)^2\frac{\partial^2}{\partial x^2}(\alpha \CFP)(t,x_n), \\
    \beta(x_{n- 1})\CFP(t,x_{n - 1}) & \simeq \beta(x_{n})\CFP(t,x_n) - \Delta x \frac{\partial}{\partial x}(\beta \CFP)(t,x_n)  + \frac{1}{2}(\Delta x)^2\frac{\partial^2}{\partial x^2}(\beta \CFP)(t,x_n),
  \end{aligned}
  \]
so that, with $\Delta x = 1$, 
\begin{equation}
\frac{\partial \CFP}{\partial t}(t,x) \simeq - \frac{\partial (F\CFP)}{\partial x}(t,x)  + \frac{1}{2} \frac{\partial^2 (D\CFP)}{\partial x^2}(t,x),
\label{eq:FK_0}
\tag{FP}
\end{equation}
where $F$ and $D$ are given by
\begin{equation}
  \label{eq:def_F_D}
  F(x) = \beta(x) C_1 - \alpha(x),\qquad  D(x) = \beta(x) C_1 + \alpha(x).
\end{equation}
The main problem with this derivation is to control the remainders of the Taylor expansions since the mesh size~$\Delta x$ is fixed. This amounts to controlling the third derivatives of $\alpha\mathscr{C}$ and $\beta\mathscr{C}$. Since $\alpha$ and $\beta$ are known, we actually only need to control the third derivative of $\mathscr{C}$. However, classical tools from the analysis of PDEs are usually used to produce \textit{a priori} regularity estimates of the solution, and sometimes of its derivatives, but rarely to state decay estimates~\cite{evans,friedman2008partial}. Moreover, the cases where such decay estimates are stated correspond to the situations where the diffusion term is bounded, which is not the case for BD as $D(x) = \mathrm{O}(x^\gamma)$.
\begin{rmk}
Even though we do not study the well-posedness of the Fokker--Planck equation~\eqref{eq:Fokker--Planck}, there exist classical solutions to Cauchy problems associated to such an equation~\cite{lebris2008existence} (see also~\cite{CS19} when the equation is not posed on the whole space but a boundary condition inspired by BD is imposed at $x=0$).
\end{rmk}

The aim of the next subsection is to reformulate the Fokker--Planck equation as another equation for which we can characterize the decay of the derivatives.

\subsection{Heuristic reformulation as a diffusion equation}
\label{sec:formal-results}

In order to control the decay of the derivatives of the solution to~\eqref{eq:FK_0}, we use a change of variables to reformulate the Fokker--Planck approximation with fixed $C_1$ as a diffusion equation without advection. The decay of the spatial derivatives of the solution of the diffusion equation can then be made precise (see Theorem~\ref{thm:regularity}).

\subsubsection{Main assumptions} 
Let us first state the assumptions we need on the coefficients $\alpha$ and $\beta$ for this analysis.
\begin{hypo}
\label{hypo:alpha-beta-prime}
The functions $\alpha$ and $\beta$ are smooth, non-decreasing  and non-negative. Moreover, there exists $0 \le \gamma \le 1/2$ such that, as $x \to +\infty$,
\begin{equation}
\forall k \in \{0,1,2\}, \qquad \alpha^{(k)}(x) = \mathrm{O}(x^{\gamma-k}), \quad \beta^{(k)}(x) = \mathrm{O}(x^{\gamma-k}).
\end{equation}
Finally, we assume that there exist $M, K_+, K_- > 0$ (which depend on $C_1$), such that, for all $x \ge M$, the function $F = \beta C_1 - \alpha$ is positive and increasing and 
\begin{equation}
\label{eq:minoration-F}
\forall\, x\ge M, \qquad K_- x^\gamma \le F(x) \le K_+ x^\gamma.
\end{equation}
\end{hypo}
A consequence of this assumption is that the functions $F$ and $D$ defined in~\eqref{eq:def_F_D} are smooth. Moreover, $F^{(k)}(x) = \mathrm{O}(x^{\gamma-k})$ for $0 \le k \le 2$, and the same estimates hold for $D$ and its derivatives.
\begin{expl}
\label{expl:alpha-beta-continus}
The functions $\alpha$ and $\beta$ associated with the coefficients of Example~\ref{expl:alpha-beta-discret} read, for $x\ge 1$,
\begin{equation}
\alpha(x) = \alpha_0 (x-1)^\gamma \exp{\left(\frac{\omega}{x^\gamma}\right)}, \qquad \beta(x) = \beta_0 x^\gamma,
\end{equation}
where $\alpha_0 = {\left(48\pi^2 V_\mathrm{at}^{-2}\right)}^{\gamma} D \exp(-E_\mathrm{v}^\mathrm{f}/(k_\mathrm{B}T))$, $\beta_0 = {\left(48\pi^2 V_\mathrm{at}^{-2}\right)}^{\gamma} D$ and $\gamma \in {\left\{\frac{1}{3}, \frac{1}{2}\right\}}$. The condition that $F = \beta C_1 - \alpha$ is positive and increasing in Assumption~\ref{hypo:alpha-beta-prime} therefore holds true for $M > \max{\left(1,\omega^{1/\gamma} \ln(\beta_0 C_1/\alpha_0)^{-1/\gamma}\right)}$. The other conditions in  Assumption~\ref{hypo:alpha-beta-prime} are easily seen to be satisfied.
\end{expl}

\subsubsection{Introducing a change of variable} 
Define the following functions for $x\ge M$:
\begin{equation}
g(x) = \frac{1}{F(x)}, \qquad G(x) = \int_M^x g(y)\,dy.
\end{equation}
Both are well defined for $x\ge M$, and smooth. Moreover, in view of~\eqref{eq:minoration-F} in Assumption~\ref{hypo:alpha-beta-prime}, $G$ is a non-negative increasing function such that $G(x) \to +\infty$ as $x\to +\infty$. We denote by $G^{-1}$ the inverse function of $G$, well defined on $[0,+\infty)$ and with values in $[M,+\infty)$. Introduce the domain
\begin{equation}
\label{eq:Z_M}
Z_M = {\left\{ (t,x) \in \mathbb{R}_+\times [M,+\infty)\ \middle| \ t \le G(x) \right\}},
\end{equation}
illustrated in Figure~\ref{fig:domain-ZM}, and the function
\begin{equation}
\label{eq:def-Qtx}
\forall\, (t,x) \in Z_M, \qquad Q(t,x) = G^{-1}(G(x)-t).
\end{equation}
Then, for a given function $\CFP_\mathrm{adv}^0 \in \mathcal{C}^1([M,+\infty))$, the function defined for all $(t,x) \in Z_M$ by
$\CFP_\mathrm{adv}(t,x) = \CFP^0_\mathrm{adv}{\left[ Q(t,x) \right]}$,
satisfies (by the methods of characteristics, see \textit{e.g.}~\cite{evans})
\begin{equation}
\forall (t,x) \in Z_M, \qquad \frac{\partial \CFP_\mathrm{adv}}{\partial t}(t,x) = - F(x) \frac{\partial \CFP_\mathrm{adv}}{\partial x}(t,x).
\end{equation}
The above equation corresponds to the dominant "advection" part of the Fokker--Planck equation~\eqref{eq:FK_0} (see the discussion at the end of this section, in particular the estimate~\eqref{eq:estimate-advection}). Let us notice that, for $t\ge 0$ and $q\ge M$, it holds $G(q) + t \ge 0$. We can therefore introduce the functions 
\begin{equation}
\label{eq:CFP2Cdiff}
\forall\, (t,q) \in \mathbb{R}_+ \times [M,+\infty),\qquad X(t,q) = G^{-1}(G(q)+t), \qquad \Cdiff(t,q) = \CFP{\left[t,X(t,q)\right]},
\end{equation}
where $\CFP$ is the solution of~\eqref{eq:Fokker--Planck}. Note that $X$ is the inverse function of the characteristic $Q(t,x)$ appearing in~\eqref{eq:def-Qtx}, \textit{i.e.} $X(t,Q(t,x)) = x$ and $Q(t,X(t,q)) = q$. Using the function $X$ will allow us to suppress the advection part in~\eqref{eq:FK_0}. Before we make this precise, let us state some useful estimates on the functions $X$ and $Q$.
\begin{lemm}
\label{lemm:GandG}
Suppose that Assumption~\ref{hypo:alpha-beta-prime} holds true. Then, there exist $\rho_1, \rho_2 > 0$ such that
\begin{equation}
\forall\, x\ge M, \qquad 0\le G(x) \le \rho_1 x^{1-\gamma}, \qquad\text{and}\qquad \forall\, x \ge 1,\qquad M\le G^{-1}(x) \le \rho_2 x^{\frac{1}{1-\gamma}}.
\end{equation}
\end{lemm}
\begin{proof}
In view of~\eqref{eq:minoration-F}, it holds, for $x\ge M$, that $0 \le G(x) \le \left(x^{1-\gamma} - M^{1-\gamma}\right)/[K_-(1-\gamma)]$, 
from which the first estimate follows. Moreover, by definition, ${\left(G^{-1}\right)}^\prime = F\circ G^{-1}$. Since $G^{-1}(x) \ge M$ for all $x\ge 0$, it holds ${\left(G^{-1}\right)}^\prime(x) \le K_+ {\left(G^{-1}(x)\right)}^{\gamma}$. Using a nonlinear generalisation of Gronwall's inequality~\cite{dragomir2003some}, we obtain $G^{-1}(x) \le {\left(M^{1-\gamma} + (1-\gamma)K_+ x \right)}^{\frac{1}{1-\gamma}}$, from which the second estimate follows.
\end{proof}
\begin{lemm}
\label{lemm:XandQ}
Fix a time $t \ge 0$. Suppose that Assumption~\ref{hypo:alpha-beta-prime} holds true. Then,
\begin{equation}
\label{eq:limit-XandQ}
\lim_{q\to +\infty} \frac{X(t,q)}{q} = 1 \qquad\text{and}\qquad \lim_{x\to +\infty} \frac{Q(t,x)}{x} = 1.
\end{equation}
\end{lemm}
\begin{proof}
Since $G(x)\to +\infty$ as $x\to +\infty$, there is $x^*$ such that $t + 1\le G(x^*)$. Therefore, $Q(t,x)$ is well-defined for all $x\ge x^*$. For any $x\ge x^*$, there exists $t_x \in [0, t]$ such that
\begin{equation}
\label{eq:taylor-Qtx}
Q(t,x) = Q(0,x) - t\partial_t Q(t_x,x) = x - t{\left(F\circ G^{-1}\right)}{\left(G(x) - t_x\right)}.
\end{equation}
Moreover, $1 \le G(x) - t_x \le G(x)$ and since $G^{-1}$ is increasing, $M \le G^{-1}{\left(G(x) - t_x\right)} \le x$. Therefore, since $F$ is increasing by Assumption~\ref{hypo:alpha-beta-prime},
\begin{equation}
\label{eq:estimate-Rtx}
F(M) \le {\left(F\circ G^{-1}\right)}{\left(G(x) - t_x\right)} \le K_+ x^\gamma,
\end{equation}
so that, in view of~\eqref{eq:taylor-Qtx}, it holds $x - t K_+ x^\gamma \le Q(t,x) \le x - tF(M)$. The conclusion follows from a squeeze theorem since $0\le \gamma \le 1/2$. A similar reasoning can be used to prove the limit of $X(t,q)/q$ as $q\to +\infty$.
\end{proof}
\begin{lemm}
\label{lemm:R_F}
Define the functions $R_{F,1}$, $R_{F,2}$ and $R_{D,1}$ as follows: for $(t,x) \in Z_M$, 
\begin{equation}
R_{F,1}(t,x) = \frac{F(Q(t,x))}{F(x)} - 1,\qquad R_{F,2}(t,x) =  \frac{F^2(Q(t,x))}{F^2(x)} - 1, \qquad R_{D,1}(t,x) = \frac{D(Q(t,x))}{D(x)} - 1.
\end{equation} 
Suppose that Assumption~\ref{hypo:alpha-beta-prime} holds true. Then, there exists a non-negative function $K \in \mathcal{C}^0(\mathbb{R}_+)$ such that
\begin{equation}
\forall\, (t,x) \in Z_M, \qquad {\left| R_{F,1}(t,x) \right|}, {\left| R_{F,2}(t,x) \right|}, {\left| R_{D,1}(t,x) \right|} \le K(t)x^{\gamma-1}.
\end{equation}
\end{lemm}
\begin{proof}
In view of~\eqref{eq:taylor-Qtx}, it holds $Q(t,x) = x - R(t,x)$ for any $(t,x) \in Z_M$, where $R(t,x) = t {\left(F\circ G^{-1}\right)}{\left(G(x) - t_x\right)}$. Fix $(t,x) \in Z_M$. Using a Taylor expansion, there exists $\zeta_{F,1},\zeta_{D,1} \in [x-R(t,x),x]$ such that $F(Q(t,x)) = F(x) - R(t,x) F^\prime(\zeta_{F,1})$ and $D(Q(t,x)) = D(x) - R(t,x) D^\prime(\zeta_{D,1})$. Moreover, there exists $\zeta_2 \in [x-R(t,x),x]$ such that $F^2(Q(t,x)) = F^2(x) - 2 R(t,x) F(\zeta_2) F^\prime(\zeta_2)$. The conclusion then follows from~\eqref{eq:minoration-F} and~\eqref{eq:estimate-Rtx}.
\end{proof}

\begin{expl}
\label{expl:alpha-beta-param}
Fix $C_1 > 0$ and consider the functions $\alpha(x) = \alpha_0 x^\gamma$ and $\beta(x) = \beta_0 x^\gamma$. These functions are asymptotically equivalent to those of Example~\ref{expl:alpha-beta-continus}. The simplicity of their expressions allows us to give analytic expressions for the functions introduced in this section. Defining $\lambda_0 = \beta_0 C_1 - \alpha_0$, it holds, for any $M > 0$,
\begin{equation}
\forall\, x\ge M,\qquad G(x) = \frac{1}{\lambda_0(1-\gamma)}{\left[x^{1-\gamma} - M^{1-\gamma}\right]},
\end{equation}
and
\begin{equation}
\forall\, y\ge 0,\qquad G^{-1}(y) = {\left[M^{1-\gamma} + \lambda_0(1-\gamma)y\right]}^{\frac{1}{1-\gamma}}.
\end{equation}
Therefore,
\begin{equation}
\forall\, (t,x) \in Z_M,\qquad Q(t,x) = {\left[x^{1-\gamma} - \lambda_0(1-\gamma) t\right]}^{\frac{1}{1-\gamma}},
\end{equation}
and
\begin{equation}
\forall\, (t,q) \in \mathbb{R}_+ \times [M,+\infty),\qquad  X(t,q) = {\left[q^{1-\gamma} + \lambda_0(1-\gamma)t\right]}^{\frac{1}{1-\gamma}}.
\end{equation}
The domain $Z_M$ is illustrated in Figure~\ref{fig:domain-ZM}, for the parameters used in~\cite{ovcharenko2003gmic++}, namely $\gamma = 1/3$, $D = 1.83\times 10^{-13}\,m^2/s$, $V_\mathrm{at} = 1.205\times 10^{-29}\, m^{3}$, $T = 823\, K$, $E_\mathrm{v}^\mathrm{f} = 1.7\, eV$ and $M = 20$.
\begin{figure}
\centering
\includegraphics[width=0.6\textwidth]{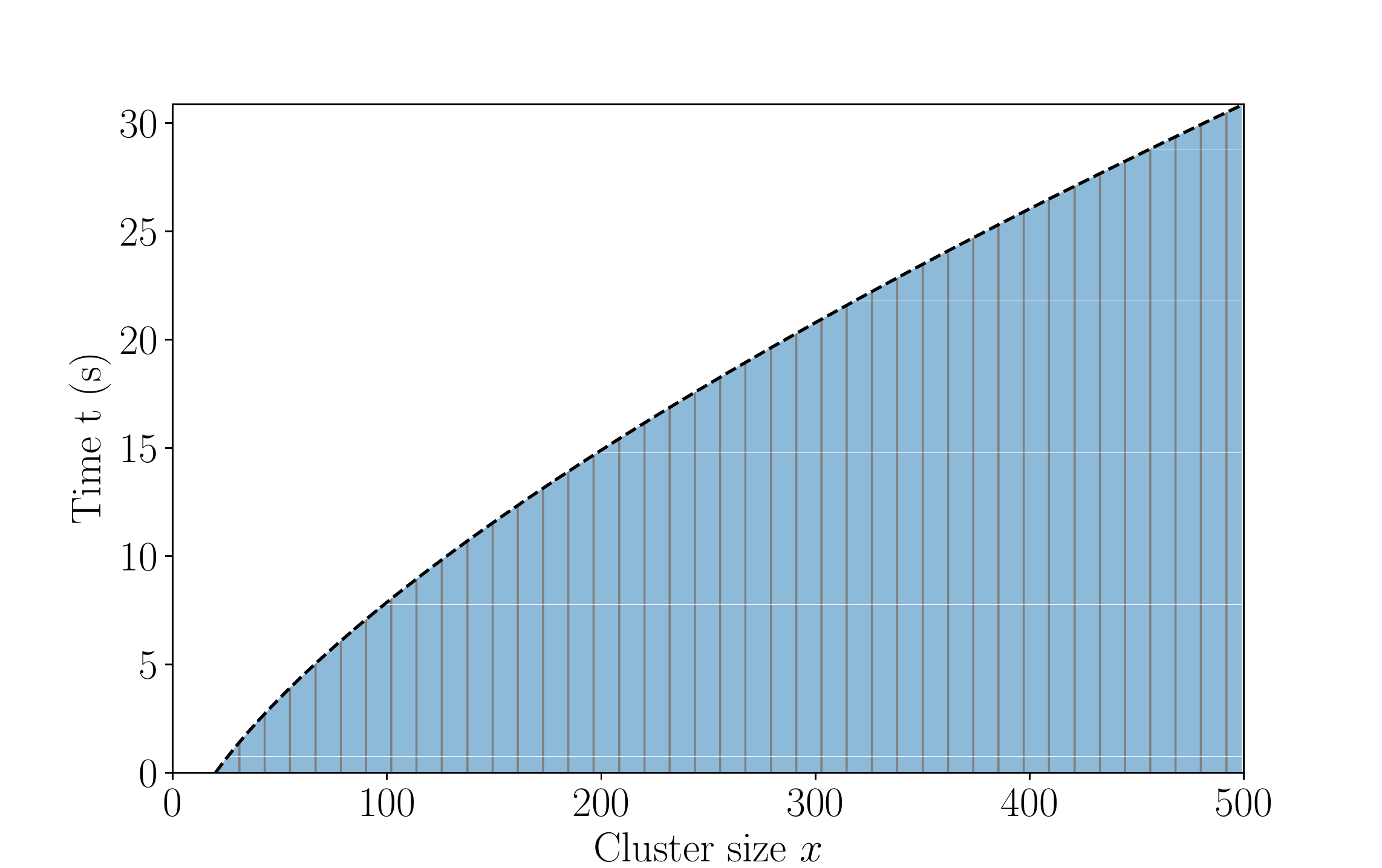}
\caption{The dashed line represents the function $G$ so that $Z_M$ is the hatched domain for the parameters of Example~\ref{expl:alpha-beta-param}.}
\label{fig:domain-ZM}
\end{figure}
\end{expl}

\subsubsection[Heuristic reformulation]{Heuristic reformulation of~\eqref{eq:FK_0}}
Let us now reformulate the Fokker--Planck equation as a diffusion equation without advection with the change of variable introduced in~\eqref{eq:CFP2Cdiff}. In order to obtain such an equation, we calculate the partial derivatives of $\Cdiff$. Let us first notice that
\begin{equation}
\forall\, (t,q) \in \mathbb{R}_+ \times [M,+\infty),\qquad  \frac{\partial X}{\partial q}{\left(t,q\right)} = \frac{F{\left[X(t,q)\right]}}{F(q)}.
\end{equation}
The ratio on the right-hand side is well defined since $F(q) > 0$ for all $q \ge M$. By the chain rule, and assuming that $\Cdiff$ is smooth (a property which will be proved later on in Section~\ref{sec:thm-decay}), it holds, for all $(t,q) \in \mathbb{R}_+ \times [M,+\infty)$,
\begin{equation}
\frac{\partial \Cdiff}{\partial q}(t,q) = \frac{F{\left[X(t,q)\right]}}{F(q)}\frac{\partial \CFP}{\partial x}(t,X(t,q)),
\end{equation}
and
\begin{equation}
\frac{\partial^2 \Cdiff}{\partial q^2}(t,q) = \frac{F^2{\left[X(t,q)\right]}}{F^2(q)} \frac{\partial^2 \CFP}{\partial x^2}(t,X(t,q)) +  \frac{F{\left[X(t,q)\right]}(F^\prime{\left[X(t,q)\right]} - F^\prime(q))}{F^2(q)} \frac{\partial \CFP}{\partial x}(t,X(t,q)).
\end{equation}
Given~\eqref{eq:limit-XandQ} and using Assumption~\ref{hypo:alpha-beta-prime}, we obtain that, in the limit $q\to +\infty$,
\begin{equation}
\frac{F{\left[X(t,q)\right]}(F^\prime{\left[X(t,q)\right]}-F^\prime(q))}{F^2(q)} = \mathrm{O}{\left(\frac{1}{q}\right)}, \qquad \frac{F^2(X(t,q))}{F^2(q)} \sim 1.
\end{equation}
We then make the assumption, which will be proved to hold as a consequence of Theorem~\ref{thm:regularity}, that, as $q\to +\infty$,
\begin{equation}
\label{eq:hypo-CFP-1}
\frac{1}{q}\frac{\partial \CFP}{\partial x}(t,X(t,q)) \ll \frac{\partial^2 \CFP}{\partial x^2}(t,X(t,q)).
\end{equation}
Then, for $q$ large,
\begin{equation}
\frac{\partial^2 \Cdiff}{\partial q^2}(t,q) \simeq \frac{\partial^2 \CFP}{\partial x^2}(t,X(t,q)).
\end{equation}
Assuming further that
\begin{equation}
\label{eq:hypo-CFP-2}
\frac{1}{q}\CFP(t,X(t,q)) \ll \frac{\partial \CFP}{\partial x}(t,X(t,q)),
\end{equation}
as $q\to +\infty$, which will also be proved to hold later on, we obtain with Assumption~\ref{hypo:alpha-beta-prime} that
\begin{equation}
\label{eq:estimate-advection}
\frac{\partial (F\CFP)}{\partial x}(t,X(t,q)) \simeq F{\left[X(t,q)\right]} \frac{\partial \CFP}{\partial x}(t,X(t,q)),
\end{equation}
and
\begin{equation}
\frac{\partial^2 (D\CFP)}{\partial x^2}(t,X(t,q)) \simeq D{\left[X(t,q)\right]}\frac{\partial^2\CFP}{\partial x^2}(t,X(t,q)) \simeq D(q)\frac{\partial^2\CFP}{\partial x^2}(t,X(t,q)).
\end{equation}
We finally consider the time derivative of $\Cdiff$ and combine the previous results in order to write the diffusion equation satisfied by $\Cdiff$. Since
\begin{equation}
\frac{\partial \Cdiff}{\partial t}(t,q) = \frac{\partial \CFP}{\partial t}(t,X(t,q)) + F{\left[X(t,q)\right]}\frac{\partial \CFP}{\partial x}(t,X(t,q)),
\end{equation}
we formally obtain that $\Cdiff$ is the solution of the following diffusion equation for large $q$:
\begin{equation}
\label{eq:diffusion-equation}
\frac{\partial \Cdiff}{\partial t}(t,q) = \frac{1}{2} \sigma^2(q) \frac{\partial^2 \Cdiff}{\partial q^2}(t,q),
\end{equation}
with diffusion coefficient
\begin{equation}
\label{eq:sigma-heuristic}
\sigma^2(q) = D(q).
\end{equation}

\subsection{Decay estimates of the solution of the diffusion equation}
\label{sec:thm-decay}

The diffusion equation~\eqref{eq:diffusion-equation} allows to relate the ODEs of BD~\eqref{eq:EDO-Vacancy} with fixed $C_1$ and the Fokker--Planck equation~\eqref{eq:Fokker--Planck}. However, before we state more precisely this result, we first need to present some results on the decay of the spatial derivatives of the solution of this diffusion equation, which will allow us to make rigorous the heuristic derivation of Section~\ref{sec:formal-results}. Consider the following Cauchy problem:
\begin{equation}
{\left\{
\begin{aligned}
&\frac{\partial \Cdiff}{\partial t} = \frac{1}{2}\sigma^2(q) \frac{\partial^2 \Cdiff}{\partial q^2},\\
&\Cdiff(0,q) = \Cdiff_0(q).
\end{aligned}
\right.}
\label{problem:Cauchy-diffusion}
\tag{P-Diff}
\end{equation}
Assumptions on the initial condition $\Cdiff_0$ will be made precise hereafter. To our knowledge, the decay of the spatial derivatives of the solutions to~\eqref{problem:Cauchy-diffusion} has never been studied in the case of an unbounded diffusion coefficient. We propose here a stochastic approach to this end, as long as $\sigma$ and $\Cdiff_0$ satisfy sufficient conditions of growth and regularity.

The main difficulty in giving decay estimates of the solution of such a problem comes from the fact that the diffusion coefficient $\sigma$ is not bounded, so that it does not satisfy some parabolic condition as in~\cite[Chapter~1.1]{friedman2012stochastic}. While Hörmander's theorem (see~\cite[Theorem 1.3]{hairer2011hormander}) ensures the existence and uniqueness of a smooth solution, for a whole class of diffusion coefficients (positive with bounded derivatives on the whole space, see Assumption~\ref{hypo:sigma-diffusion}), it does not provide decay estimates on the solution. In this section, we first discuss the form of $\sigma$ as defined in~\eqref{eq:sigma-heuristic}, before stating decay estimates on the solution of~\eqref{problem:Cauchy-diffusion}.

\subsubsection{Characterization of the coefficient $\sigma$}

Since we want to prove the correctness of the Fokker--Planck approximation in the asymptotic limit $M\to +\infty$, where $M$ is the minimal size of a cluster, we only need to control the spatial derivatives of the solution when the space variable goes to infinity. While the expression~\eqref{eq:sigma-heuristic} holds true only on $[M,+\infty)$, the use of stochastic tools requires $\sigma$ to be defined on the whole space. In order to guarantee the existence and uniqueness of the solution to the problem~\eqref{problem:Cauchy-diffusion}, we require that $\sigma \in \mathcal{C}^\infty(\mathbb{R}_+\times\mathbb{R})$ is positive with bounded derivatives (which is guaranteed by Assumption~\ref{hypo:sigma-diffusion} below). Let us now give an expression of $\sigma$ in a simple case, which will give us a useful guideline for the following.
\begin{expl}
\label{expl:sigma-simple}
Fix $C_1 > 0$ and consider the coefficients $\alpha$ and $\beta$ defined in Example~\ref{expl:alpha-beta-continus}. Then, for all $q\ge 1$,
\begin{equation}
D(q) = \beta_0 C_1 q^\gamma + \alpha_0 (q-1)^\gamma\exp{\left(\frac{\omega}{q^\gamma}\right)} = q^\gamma {\left(\beta_0 C_1 + \alpha_0 {\left(1-\frac{1}{q}\right)}^\gamma \exp{\left(\frac{\omega}{q^\gamma}\right)}\right)}.
\end{equation}
Therefore, $\sigma$ writes as $\sigma(q) = \sigma_0(q)\sigma_\mathrm{b}(q)$ where $\sigma_0(q) = q^{\gamma/2}$ and $\sigma_\mathrm{b}$ is bounded with bounded derivatives. In fact, the derivative of order $k$ of $\sigma_\mathrm{b}$ asymptotically decays as $q^{-k}$. The function $\sigma_0$ represents the main difficulty of our problem since it is not bounded.
\end{expl}
As suggested in Example~\ref{expl:sigma-simple}, and in view of Assumption~\ref{hypo:alpha-beta-prime}, we assume that $\sigma$ can be written as $\sigma(q) = \sigma_0(q)\sigma_\mathrm{b}(q)$, with $\sigma^2_0(q) = q^\gamma$ for $q \ge M$, and $\sigma_\mathrm{b}$ a smooth positive bounded function on $[M,+\infty)$. In fact, in view of Assumption~\ref{hypo:alpha-beta-prime}, $\sigma_\mathrm{b}$ is automatically bounded since there exist $K_{D,-}, K_{D,+} \in \mathbb{R}_+^*$ such that $K_{D,-} q^\gamma \le D(q) \le K_{D,+} q^\gamma$ for $q\ge M$, so that
\begin{equation}
\forall q\ge M, \qquad K_{D,-} \le \sigma_\mathrm{b}^2(q) := \frac{D(q)}{q^\gamma} \le K_{D,+}.
\end{equation}
Moreover, in view of Assumption~\ref{hypo:alpha-beta-prime}, we also obtain estimates on the derivatives of $\sigma$, up to second order. In the next section, in order to obtain general results, we assume bounds on derivatives of all order for $\sigma_\mathrm{b}$.

\subsubsection[Decay estimates of the solution to the diffusion equation]{Decay estimates of the solution of~\eqref{problem:Cauchy-diffusion}}

We present in this section two results concerning the decay of the solution of the Cauchy problem~\eqref{problem:Cauchy-diffusion}. Let us emphasize that the results are stated on the whole space $\mathbb{R}_+ \times \mathbb{R}$ for the Cauchy problem~\eqref{problem:Cauchy-diffusion}. Our assumptions on $\sigma$ are the following.
\begin{hypo}
\label{hypo:sigma-diffusion}
The diffusion coefficient $\sigma$ is a smooth positive function with bounded derivatives. Moreover,
$\sigma(q) = \sigma_0(q) \sigma_\mathrm{b}(q)$,
with $\sigma_0(q) = |q|^{\gamma/2}$ for $|q|\geq 1$ for some $0\le\gamma\le 1/2$; and $\sigma_\mathrm{b}$ is a bounded smooth function with bounded derivatives for which there exist $\delta_-, \delta_+ > 0$ such that $\delta_- \le \sigma_\mathrm{b}(q) \le \delta_+$. Finally, for any $n\ge 1$, there exists $S_n \in \mathbb{R}_+$ such that ${\left| \sigma^{(n)}(q) \right|} \le S_n {\left|q\right|}^{\gamma/2-n}$ for $|q|\geq 1$.
\end{hypo}
Note that the function $\sigma$ of Example~\ref{expl:sigma-simple} satisfies Assumption~\ref{hypo:sigma-diffusion}. We also need assumptions on the initial condition $\Cdiff_0$.
\begin{hypo}
\label{hypo:C0}
The initial condition $\Cdiff_0$ is a smooth bounded function. Moreover, for all $n\ge 1$, there is a constant $R_n \in \mathbb{R}_+$ such that ${\left\| \Cdiff_0^{(n)} \sigma_0^n \right\|}_{\mathcal{C}^0} \le R_n$.
\end{hypo}
Note that functions in $\mathcal{S}(\mathbb{R})$, the Schwartz space of rapidly decreasing functions, satisfy Assumption~\ref{hypo:C0}. We are then in position to prove the following result.
\begin{thm}
\label{thm:regularity}
Fix an initial condition $\Cdiff_0$ satisfying  Assumption~\ref{hypo:C0}, and suppose that Assumption~\ref{hypo:sigma-diffusion} holds. Then, there exists a unique classical solution of~\eqref{problem:Cauchy-diffusion}, which is smooth and bounded. Moreover, for all $n\ge 1$, there exists  a non-negative function $K_n \in \mathcal{C}^0(\mathbb{R}_+)$ such that
\begin{equation}
\label{eq:decay-estimates}
\forall\, (t,q) \in \mathbb{R}_+ \times \mathbb{R},\qquad {\left| \frac{\partial^n \Cdiff}{\partial q^n}(t,q) \right|} \le K_n(t){\left|q\right|}^{-n\gamma/2}.
\end{equation}
\end{thm}
Notice that the bound depends on time through the functions $K_n$. Typically $K_n$ grows exponentially in time. Therefore, this result is useful for estimates at finite times. The proof of this result, given in Appendix~\ref{sec:appendix-decay}, relies on stochastic techniques, the fundamental solution of~\eqref{problem:Cauchy-diffusion} being interpreted as the law of a stochastic process. Let us emphasize that the assumptions stated in~\eqref{eq:hypo-CFP-1} and~\eqref{eq:hypo-CFP-2} hold true in view of Theorem~\ref{thm:regularity} and Lemma~\ref{lemm:XandQ}. 
\begin{rmk}
Theorem~\ref{thm:regularity} can in fact be extended to $0\le \gamma < 1$. Moreover, a use of Malliavin's calculus~\cite{nualart2006malliavin} allows one to conclude for $\gamma = 1$ if $\sigma = \sigma_0$. Nonetheless, in order to be consistent with the remainder of our work, we limit ourselves to $0\le\gamma\le 1/2$.
\end{rmk}
\begin{rmk}
Cerrai~\cite[Chap. 1.5]{cerrai2001second} proves the existence of a unique smooth classical solution of~\eqref{problem:Cauchy-diffusion} assuming only $\sigma \in \mathcal{C}^3(\mathbb{R})$ with polynomial growth and an initial condition $\Cdiff_0 \in \mathcal{C}_b^2(\mathbb{R})$. Therefore, since we are only interested in the third spatial derivative of the solution of~\eqref{problem:Cauchy-diffusion}, and as a careful inspection of the proof in Appendix~\ref{sec:appendix-decay} shows, we could relax some assumptions on $\sigma$ and $\Cdiff_0$ and limit the assumptions on their derivatives up to order~$3$.
\end{rmk}

\subsection{Relating Becker--D\"oring equations and their Fokker--Planck approximation}
\label{sec:Diff2CDandFP}

We are now in position to rigorously relate the Fokker--Planck approximation~\eqref{eq:FK_0} and the BD equations~\eqref{eq:EDO-Vacancy}. This section is divided into two parts. We first present a result relating the diffusion problem~\eqref{problem:Cauchy-diffusion} and the Fokker--Planck approximation~\eqref{eq:FK_0} (see Section~\ref{sec:Diff2FP}); and then a result relating the diffusion problem and the BD equations (see Section~\ref{sec:Diff2CD}). We discuss the domain on which such results hold true and quantify the error arising from the approximations.

\subsubsection{From the diffusion equation to the Fokker--Planck equation}
\label{sec:Diff2FP}

Let us first relate the diffusion equation and the Fokker--Planck equation by proving that the solution of the diffusion equation satisfies up to a change of variable the Fokker--Planck equation, up to an error term whose magnitude we quantify.
\begin{thm}
\label{thm:Diff2FP}
Suppose that Assumptions~\ref{hypo:alpha-beta-prime}, \ref{hypo:sigma-diffusion} and~\ref{hypo:C0} hold true and denote by $\Cdiff \in \mathcal{C}^{\infty}(\mathbb{R}_+ \times \mathbb{R})$ the solution of the diffusion problem~\eqref{problem:Cauchy-diffusion} with initial condition $\Cdiff_0$. Define $\CFP \in \mathcal{C}^\infty(Z_M)$ as
\begin{equation}
\forall\, (t,x) \in Z_M,\qquad \CFP(t,x) = \Cdiff(t,Q(t,x)),
\end{equation}
where $Q$ is introduced in~\eqref{eq:def-Qtx}. Then,
\begin{equation}
\forall\, (t,x) \in Z_M,\qquad \frac{\partial\CFP}{\partial t} = -\frac{\partial (F\CFP)}{\partial x} + \frac{1}{2}\frac{\partial^2 (D\CFP)}{\partial x^2} +  R_{\Cdiff}(t,x),
\end{equation}
where there exists a non-negative function $K\in\mathcal{C}^0(\mathbb{R}_+)$ such that
\begin{equation}
\forall\, (t,x) \in Z_M,\qquad {\left| R_{\Cdiff}(t,x) \right|} \le K(t) x^{\gamma-1}.
\end{equation}
\end{thm}
The estimates we obtain for $R_\Cdiff$ on $Z_M$ show that the error arising from the reformulation of~\eqref{eq:FK_0} as~\eqref{problem:Cauchy-diffusion} becomes smaller as the size of the clusters increases. In practice, since we do not obtain lower bounds on ${\left|R_\Cdiff\right|}$, we however cannot ensure whether $R_\Cdiff$ is negligible in front of $\partial_x (F\mathscr{C})$ and $\partial_{xx}(D\mathscr{C})$. 

\subsubsection{From the diffusion equation to the Becker--D\"oring equations}
\label{sec:Diff2CD}

We now give a result relating the diffusion equation and the BD equations, up to a small error term which can be quantified. Since we work with discrete variables, we consider the discrete version $\mathcal{Z}_M$ of the space $Z_M$ defined in~\eqref{eq:Z_M}:
\begin{equation}
\mathcal{Z}_M = {\left\{ (t,n) \in \mathbb{R}_+\times\mathbb{N}\ \middle| \ n \ge M\ \text{and}\ t \le G(n) \right\}},
\end{equation}
on which the following approximation holds.
\begin{thm}
\label{thm:Diff2CD}
Suppose that Assumptions~\ref{hypo:alpha-beta-prime}, \ref{hypo:sigma-diffusion} and~\ref{hypo:C0} hold true and denote by $\Cdiff \in \mathcal{C}^{\infty}(\mathbb{R}_+ \times \mathbb{R})$ the solution of the diffusion problem~\eqref{problem:Cauchy-diffusion} with initial condition $\Cdiff_0$. Fix an integer $n_0 \ge M$ and a time $t_{n_0}^* = G(n_0)$. Consider the sequence of smooth functions $\widehat{C} = {\left(\widehat{C}_{n_0},\cdots,\widehat{C}_n,\cdots\right)}$ defined, for all $n\ge n_0$, by
\begin{equation}
\forall\, t \in [0, t_{n_0}^*],\qquad \Ccluster_n(t) = \Cdiff{\left(t,Q(t,n)\right)},
\end{equation}
where $Q$ is defined in~\eqref{eq:def-Qtx}. Then, there exists a non-negative function $K\in\mathcal{C}^0(\mathbb{R}_+)$ such that, for all $n\ge n_0$,
\begin{equation}
\label{problem:P1-diffusion}
\forall t \in [0,t_{n_0}^*], \qquad \frac{d\widehat{C}_n}{dt} = \beta_{n-1}\widehat{C}_{n-1} C_1 - {\left(\beta_n C_1 + \alpha_n\right)} \widehat{C}_n + \alpha_{n+1}\widehat{C}_{n+1} + R_n,
\end{equation}
where the remainder $R_n \in \mathcal{C}^0([0,t^*_{n_0}],\mathbb{R})$ satisfies
\begin{equation}
\forall\, (t,n) \in \mathcal{Z}_{n_0},\qquad {\left| R_n(t) \right|} \le \frac{K(t)}{n^{\gamma/2}}.
\end{equation}
\end{thm}
This result gives us an estimate of the error due to the approximation based on the diffusion equation. It shows that the approximation improves when the sizes of the cluster increase. Note that the approximation is only valid for a limited time $t_{n_0}^*$, depending on the minimal size $n_0$ of the clusters for which the Fokker--Planck approximation is considered. As the minimal size $n_0$ grows, the approximation stays valid for longer times. Nevertheless, in view of the splitting introduced in Section~\ref{sec:splitting}, the approximation only needs to hold true on a limited time step $\Delta t$. This gives us the minimal size one can chose, which is characterized by $n_0 = \max{\left(\lceil M\rceil ,G^{-1}(\Delta t)\right)}$. In practice $\lim_{\Delta t\to 0} G^{-1}(\Delta t) = M$, therefore the limitation on $n_0$ is characterized by the real number $M$ which ensures that $F$ is positive on $[M,+\infty)$.

\appendix

\section{Alternative proof for the uniqueness of BD}
\label{sec:appendix-well-posedness}

We show the uniqueness of the solution with an argument based on the dissipativity of the evolution operator. We start by studying the operator $A(b)$ when the parameter $b$ is fixed to a constant value in $\mathbb{R}_+$. Let us emphasize that the non-negativity of $b$ is crucial for proving the dissipativity of~$A(b)$ on~$\mathcal{H}$. We introduce the unbounded linear operators $\Alin^\alpha$ and $\Alin^\beta$ such that $A(b) = \Alin^\alpha + b \Alin^\beta$, and consider the domain 
$$
D(A(b)) = {\left\{ u \in \mathcal{H} \ \middle| \ \sum_{n\ge 1} (A(b) u)_n^2 < +\infty \right\}},
$$ 
which is dense since it contains $c_{00}$, the space of sequences which have only finitely many nonzero elements.
It is easy to check that $A(b)$ is closed (see~\cite{Terrier}).

Let us next introduce a sequence which naturally arises in the following analysis. Fix $\lambda > 0$ and define the sequence $(\delta_n^\lambda)_{n\ge 1}$ as 
\begin{equation}
\label{eq:delta-lambda}
\delta^\lambda_1 = \lambda + \beta_1 b,
\qquad
\delta^\lambda_n = \lambda + \beta_n b + \alpha_n - \cfrac{1}{4}\cfrac{(\beta_{n-1} b +\alpha_n)^2}{\delta^\lambda_{n-1}},
\end{equation}
which is well defined as long as $\delta_n^\lambda > 0$ for $n\ge 1$. We can then state the following result, valid in a setting more general than the one provided by Assumption~\ref{hypo:alpha-beta-discrete}, and which coincides in fact with the setting of~\cite[Theorem~2.1]{Laurencot2002becker}.
\begin{lemm}
\label{lemm:delta}
Suppose that there exists $B >0$ such that the nonnegative coefficients $\alpha_n,\beta_n$ satisfy $\alpha_{n+1}-\alpha_n \leq B$ and $\beta_{n+1}-\beta_n \geq -B$ for any $n \geq 1$. Then, for
\begin{equation}
\label{eq:lambda-0}
\lambda_b = \frac{1}{2}\max{\left( B{\left(1+b\right)}, \alpha_2\right)} > 0,
\end{equation}
the sequence $\delta^{\lambda_b}$ is well defined and satisfies the following lower bound:
\begin{equation}
\forall\, n \in \mathbb{N}^*,\qquad \delta^{\lambda_b}_n \ge \frac{1}{2}(\alpha_{n+1} + \beta_{n}b) > 0.
\end{equation}
\end{lemm}
\begin{proof}
We proceed by induction. For $n = 1$, it holds $\delta^{\lambda_b}_1 = \lambda_b + \beta_1 b \ge (\alpha_2+\beta_1 b)/2 > 0$.
Assume now that $\delta^{\lambda_b}_n \ge \frac{1}{2}(\alpha_{n+1} + \beta_{n} b) > 0$ for some integer $n\ge 1$. Since $(\delta^{\lambda_b}_n)^{-1} \le \frac{2}{\alpha_{n+1}+\beta_n b}$, one obtains
\begin{equation}
\begin{aligned}
\delta^{\lambda_b}_{n+1} &\ge \lambda_b + (\alpha_{n+1} + \beta_{n+1} b) - \frac{1}{2}(\alpha_{n+1} + \beta_n b) \\
& = \lambda_b + \frac{1}{2}(\alpha_{n+2} + \beta_{n+1} b) + \frac{1}{2}(\alpha_{n+1} - \alpha_{n+2}) + \frac{b}{2}(\beta_{n+1}-\beta_n) \geq \frac{1}{2}(\alpha_{n+2} + \beta_{n+1} b) + {\left(\lambda_b - \frac{B}{2}(1+b) \right)},
\end{aligned}
\end{equation}
which concludes the proof.
\end{proof}
\begin{prop}
\label{prop:A-dissipative}
Suppose that Assumption~\ref{hypo:alpha-beta-discrete} holds. Then, the operator $A(b) - \lambda_b I$ is dissipative.
\end{prop}
\begin{proof}
  Consider the sequence space $c_{00}^n = {\left\{u \in \mathbb{R}^{\mathbb{N}^*}\  \middle| \ \forall\, k > n, u_k = 0 \right\}}$, composed of sequences whose non-vanishing coefficients are the first $n$ components. We prove by induction that, for all $u\in c_{00}$, it holds $\langle (A(b) - \lambda_b I)u,u\rangle \le 0$, and then conclude by a density argument. More precisely, we consider the following statement (for the sequence $\delta^{\lambda_b}_n$ defined in~\eqref{eq:delta-lambda}): 
\begin{equation}
P(n): \qquad \forall u \in c_{00}^n, \qquad  \langle (A(b)  - \lambda_b I) u , u \rangle \le - \delta^{\lambda_b}_n u_n^2 \le 0.
\end{equation}
This amounts to proving that the operator $P_n(A(b) - \lambda_b I) P_n$ is dissipative, where $P_n$ is the projection onto $c_{00}^n$ defined as $P_nu = (u_1,\cdots,u_n,0,\cdots)$. For the induction basis $n = 1$, one simply notes that, with $u = u_1 e_1 \in c_{00}^1$,
\begin{equation}
\langle (A(b)  - \lambda_b I)u , u \rangle = - (\lambda_b + 2\beta_1 b) u_1^2  = -(\lambda_b + \beta_1 b)u_1^2 - \beta_1 b u_1^2 \le - \delta^{\lambda_b}_1 u_1^2 \le 0.
\end{equation}
Assume now that $P(n)$ holds for some integer $n \ge 1$. Consider $u \in c_{00}^{n+1}$, with $u = \sum_{k=1}^{n+1} u_k e_k = \widehat{u_n} + u_{n+1} e_{n+1}$. Then,
\begin{equation}
\begin{aligned}
\langle (A(b)  - \lambda_b I)u , u \rangle &= \langle (A(b)  - \lambda_b I)\widehat{u_n} , \widehat{u_n} \rangle + \langle (A(b) - \lambda_b I) e_n , e_{n+1} \rangle u_{n+1}u_n \\
&\qquad  + \langle (A(b) - \lambda_b I) e_{n+1} , e_n \rangle u_{n+1}u_n  + \langle (A(b) - \lambda_b I) e_{n+1} , e_{n+1} \rangle u_{n+1}^2 \\
& = \langle (A(b) - \lambda_b I) \widehat{u_n} , \widehat{u_n} \rangle + (\beta_n b + \alpha_{n+1})u_{n+1}u_n - (\lambda_b + \beta_{n+1} b + \alpha_{n+1})u_{n+1}^2  \\
&\le -\delta^{\lambda_b}_n u_n^2 + (\beta_n b + \alpha_{n+1})u_{n+1}u_n - (\lambda_b + \beta_{n+1} b + \alpha_{n+1})u_{n+1}^2,
\end{aligned}
\label{AA*}
\end{equation}
using $P(n)$ with $\widehat{u_n} \in c_{00}^n$. Let $\mathcal{R}(u_n)$ be the second-order polynomial function defined for $u_{n+1}$ fixed as
\begin{equation}
\mathcal{R}(u_n) = -\delta^{\lambda_b}_n u_n^2 + (\beta_n b + \alpha_{n+1})u_{n+1}u_n - (\lambda_b + \beta_{n+1} b + \alpha_{n+1})u_{n+1}^2.
\end{equation}
Since $-\delta^{\lambda_b}_n \le 0$ (in view of Lemma~\ref{lemm:delta}), the maximum of $\mathcal{R}$ is attained for  $u_n^\mathrm{max}  = \frac{(\beta_n b + \alpha_{n+1})u_{n+1}}{2\delta^{\lambda_b}_n}$, so that
\begin{equation}
\mathcal{R}(u_n) \le \mathcal{R}{\left(u_n^\mathrm{max}\right)} = - (\lambda_b + \beta_{n+1} b + \alpha_{n+1})u_{n+1}^2 + \frac{1}{4}\frac{{\left(\beta_n b + \alpha_{n+1}\right)}^2}{\delta^{\lambda_b}_n}u_{n+1}^2 = -\delta^{\lambda_b}_{n+1}u_{n+1}^2.
\end{equation}
This shows that $P(n+1)$ holds. Since $c_{00}$ is dense in $\mathcal{H}$ and $A(b) - \lambda_b I$ is a closed operator, we can conclude that $A(b) - \lambda_b I$ is a dissipative operator.
\end{proof}

We are now in position to prove the uniqueness of the solution to the Cauchy problem~\eqref{problem:P3}.

\begin{proof}[Proof of Theorem~\ref{thm:long-time} -- Uniqueness]
  Let $u$ and $v$ be two solutions of~\eqref{problem:P3} with initial condition $u^0$. Then, $u-v$ is solution of
\begin{equation}
\frac{d(u-v)}{dt} = A(u_1)(u-v) - (A(v_1)-A(u_1))v =  A(u_1)(u-v) - (v_1-u_1)\Alin^\beta v,
\end{equation}
so that
\begin{equation}
\frac{d\|u-v\|^2}{dt} = 2 \langle A(u_1)(u-v), u-v \rangle + 2 (v_1 - u_1) \langle \Alin^\beta v, u-v \rangle.
\end{equation}
Then, in view of Proposition~\ref{prop:A-dissipative}, and since $u_1(t) \le Q_0$ for all $t\ge 0$, it holds $\langle A(u_1)(u-v), u-v \rangle \le \lambda_{Q_0} \|u-v\|^2$. Moreover, using a Cauchy-Schwarz inequality, the equalities $Q(u(t)) = Q(v(t)) = Q_0$ and Lemma~\ref{lemm:estimates-A-AA} below, 
\begin{equation}
(v_1 - u_1) \langle \Alin^\beta v, u-v \rangle \le |u_1-v_1| R(Q_0) \|u-v\| \le R(Q_0) \|u-v\|^2.
\end{equation}
Therefore,
\begin{equation}
\label{eq:estimate-d-u-v-2}
\frac{d\|u-v\|^2}{dt} \le (\lambda_{Q_0} + R(Q_0)) \|u-v\|^2.
\end{equation}
Since $u(0) = v(0)$, we conclude that $u(t) - v(t) = 0$ for all $t\ge 0$ by a Gronwall inequality.
\end{proof}

\section[Proofs related to the splitting of the dynamics]{Proofs of the results of Section~\ref{sec:splitting}}
\label{sec:appendix-splitting}

This section is organized as follows. In Section~\ref{sec:appendix-sub2}, we use the fact that the linear operator $A(b)$ is dissipative for every $b\ge 0$, to obtain estimates on the sub-dynamics~\eqref{eq:sub-dynamic2}. In Section~\ref{sec:appendix-splitting-uprime} we  give estimates on elements of~$\mathcal{Q}$ and prove Proposition~\ref{prop:uprime}. In Section~\ref{sec:appendix-splitting-proof} we prove the convergence of the splitting.

\subsection[Estimates on the second subdynamics]{Estimates on the subdynamics~\eqref{eq:sub-dynamic2}}
\label{sec:appendix-sub2}

In this section we consider the operator $A(b)$ when the parameter $b$ is fixed to a constant value in $\mathbb{R}_+$. Note that the linear BD equations can be rewritten as the following Cauchy problem:
\begin{equation}
\frac{du}{dt} = A(b) u, \qquad u(0) = u^0.
\label{problem:P1}
\tag{LB}
\end{equation}
Since $A(b)$ is dissipative, standard results of the theory of semigroups show that~\eqref{problem:P1} has a unique classical solution in $\mathcal{C}^0(\mathbb{R}_+,D(A(b)))\cap \mathcal{C}^1(\mathbb{R}_+,\mathcal{H})$ when $u^0 \in D(A(b))$ (see~\cite[Chapter 4, Theorem 1.3]{pazy2012semigroups}). This is summarized in the following result.
\begin{prop}
The operator $A(b)$ is the infinitesimal generator of a strongly continuous semigroup $(T_b(t))_{t\in \mathbb{R}_+}$. For all $u^0 \in D(A(b))$, the problem \eqref{problem:P1} therefore has a unique solution $u\in \mathcal{C}^0(\mathbb{R}_+,D(A(b)))\cap \mathcal{C}^1(\mathbb{R}_+,\mathcal{H})$ defined as $u(t) = T_b(t)u^0$ for all $t \ge 0$. Moreover the following a priori estimates hold true:
\begin{equation}
\forall\, t \ge 0, \qquad \|u(t)\| \le \e^{\lambda_b t} {\left\|u^0\right\|}, \qquad  {\left\| \frac{du}{dt}(t) \right\|} =  \| A(b) u(t) \| \le \e^{\lambda_b t} {\left\| A(b) u^0 \right\|} .
\end{equation}
\label{prop:LCD}
\end{prop}

We next give \textit{a priori} estimates on the sub-dynamics~\eqref{eq:sub-dynamic2} which are useful for the proof of Proposition~\ref{prop:splitting}. Introducing the projection $\Pi$ such that $\Pi u = (0,u_2,u_3,\cdots)$ for all $u = (u_i)_{i\ge 1} \in \mathcal{H}$, we can define the operator $A^\Pi(u_1)$ as $A^\Pi(u_1) = \Pi A(u_1)$. The sub-dynamics~\eqref{eq:sub-dynamic2} can then be written compactly as the following linear evolution problem:
\begin{equation}
\label{problem:sub-linear}
\frac{du}{dt} = A^\Pi{\left(u_1^0\right)}u, \qquad u(0) = u^0.
\end{equation}
The following results are direct consequences of Proposition~\ref{prop:LCD} and~\cite[Chapter~4, Corollary~2.5]{pazy2012semigroups}.
\begin{prop}
\label{prop:splitting-semigroup}
Fix $u_1 \ge 0$ and suppose that Assumption~\ref{hypo:alpha-beta-discrete} holds. Then, the operator $A^\Pi(u_1)$ is the infinitesimal generator of a strongly continuous semigroup $(T_{u_1}(t))_{t\in \mathbb{R}_+}$. The problem~\eqref{eq:sub-dynamic2} therefore has a unique solution $u \in \mathcal{C}^0(\mathbb{R}_+,D(A(u_1)))\cap\mathcal{C}^1(\mathbb{R}_+,\mathcal{H})$ for all $u^0 \in D(A(u_1))$, and $u(t) = T_{u_1}(t)u^0$ for all $t \ge 0$. Moreover, there exists $\lambda_{u_1} \ge 0$ such that the following a priori estimates hold true:
\begin{equation}
\forall\, t \ge 0, \qquad {\left\|u(t)\right\|} \le \e^{\lambda_{u_1} t} {\left\|u^0\right\|} \qquad \text{and} \qquad {\left\| \frac{du}{dt}(t) \right\|} =  {\left\| A^\Pi(u_1) u(t) \right\|} \le \e^{\lambda_{u_1} t} {\left\| A^\Pi(u_1) u^0 \right\|} .
\end{equation}
Finally, fix $T > 0$ and consider $f \in \mathcal{C}^1([0,T],\mathcal{H})$. Then, the problem 
\begin{equation}
\frac{du}{dt} = A^\Pi(u_1) u + f, \qquad u(0) = u^0,
\label{problem:P1bis}
\end{equation}
has a unique classical solution $u \in \mathcal{C}^0([0,T],D(A(u_1)))\cap\mathcal{C}^1([0,T],\mathcal{H})$ defined as
\begin{equation}
\forall\, t \in [0,T], \qquad u(t) = T_{u_1}(t)u^0 + \int_0^t T_{u_1}(t-s)f(s)\,ds.
\end{equation}
\end{prop}
Finally, in order to prove the convergence of the splitting, we need estimates in $\mathcal{Q}$ of the solutions of~\eqref{eq:sub-dynamic2}.
\begin{lemm}
\label{lemm:splitting-remains-Q}
Fix an initial condition $u^0 \in \mathcal{Q}$ and $b \ge 0$. Suppose that Assumptions~\ref{hypo:alpha-beta-discrete} and~\ref{hypo:alpha-beta-gamma} hold true. Then, the unique classical solution of the second sub-dynamics $u: t\mapsto \chi_t^b{\left(u^0\right)} \in \mathcal{C}^0([0,T],D(A(b)))\cap\mathcal{C}^1([0,T],\mathcal{H})$ remains in $\mathcal{Q}$. Moreover,
\begin{equation}
Q(u(t)) \le Q(u(0))\exp(2b K t).
\end{equation}
\end{lemm}
\begin{proof}
We first note that all components of $u$ remain non-negative since the dynamics $t\mapsto \chi_t^{b}(u)$ is in fact the Kolmogorov forward equation of a Markov jump process~\cite{terrier2017coupling}. Then, using the regularity results from Theorem~\ref{thm:long-time}, we see that $Q(u(t)) = b+ \sum_{n\ge 2} n u_n(t)$ is well defined, continuously differentiable, and
\begin{equation}
\frac{d}{dt}\big[ Q(u(t)) \big] = 2\beta_1 b^2 + \sum_{n\ge 2} b \beta_n u_n(t) - \sum_{n\ge 2}\alpha_n u_n(t) - \alpha_2 u_2(t) 
\le 2 b K {\left(b + \sum_{n \ge 2} n^\gamma u_n(t)\right)} \le 2bK Q(t).
\end{equation}
The claimed estimate then follows from a Gronwall inequality.
\end{proof}

\subsection{Some estimates on elements of $\mathcal{Q}$}
\label{sec:appendix-splitting-uprime}

We state in this section estimates for elements of the set $\mathcal{Q}$ introduced in~\eqref{eq:def-Q}.
\begin{lemm}
\label{lemm:estimates-A-AA}
Fix $Q_* \in \mathbb{R}_+$ and suppose that Assumptions~\ref{hypo:alpha-beta-discrete} and~\ref{hypo:alpha-beta-gamma} hold true. Then, there exists $R(Q_*) \in \mathbb{R}_+$ such that, for any $w\in \mathcal{Q}$ with $Q(w) \le Q_*$,
\begin{equation}
{\left\|\Alin^\alpha w \right\|}, {\left\|\Alin^\beta w \right\|}, {\left\|{\left(\Alin^\alpha\right)}^2 w \right\|}, {\left\|\Alin^\alpha \Alin^\beta w \right\|}, {\left\|\Alin^\beta \Alin^\alpha w \right\|}, {\left\|{\left(\Alin^\beta\right)}^2 w \right\|} \le R(Q_*).
\end{equation}
\end{lemm}
\begin{proof}
Fix $w \in \mathcal{Q}$ such that $Q(w) \le Q_*$. Note first that $0\le w_n \le Q_*/n$. Then, using Assumption~\ref{hypo:alpha-beta-discrete} to bound $\alpha_n$ as $0\le \alpha_n \le Kn$ for $n\ge 1$,
\begin{equation}
{\left({A}_\mathrm{L}^\alpha w\right)}_1^2 = {\left(\sum_{n\ge 2} \alpha_n w_n + \alpha_2 w_2\right)}^2 \le {\left(2K\sum_{n\ge 2} n w_n\right)}^2 \le (2K Q_*)^2.
\end{equation}
Moreover, 
\begin{equation}
\sum_{n\ge 2} {\left({A}_\mathrm{L}^\alpha w\right)}_n^2 = \sum_{n\ge 2} {\left(\alpha_{n+1} w_{n+1} - \alpha_n w_n \right)}^2 \le 4 \sum_{n\ge 2} {\left(\alpha_{n} w_{n} \right)}^2 \le 4 \sum_{n\ge 2} \alpha_n^2 \frac{Q_*}{n^2} (n w_n) \le 4 K^2 Q_* \sum_{n\ge 2} n w_n.
\end{equation}
Therefore,
\begin{equation}
\sum_{n\ge 2} {\left({A}_\mathrm{L}^\alpha w\right)}_n^2 \le 4K^2 Q_* \sum_{n\ge 2} n w_n \le 4 K^2 Q_*^2,
\end{equation}
which gives us ${\left\|{A}_\mathrm{L}^\alpha w\right\|} \le 2\sqrt{2}KQ_*$. A similar reasoning can be used to bound $ {\left\|{A}_\mathrm{L}^\beta w \right\|}$. Let us next consider ${\left\|{\left({A}_\mathrm{L}^\alpha\right)}^2 w \right\|}$. Since $\alpha$ is non-decreasing, it holds
\begin{equation}
{\left|{\left({A}_\mathrm{L}^\alpha {A}_\mathrm{L}^\alpha w\right)}_1 \right|} = {\left| \sum_{n\ge 2} \alpha_n {\left( \Alin^\alpha w \right)}_n + \alpha_2 {\left( \Alin^\alpha w\right)}_2\right|} \le 2 \sum_{n\ge 2} \alpha_n{\left|\alpha_{n+1} w_{n+1} - \alpha_n w_n \right|} \le 4 \sum_{n\ge 2} \alpha_n^2 w_n.
\end{equation}
In view of Assumption~\ref{hypo:alpha-beta-gamma}, it holds, with $0\le \gamma \le 1/2$,
\begin{equation}
\sum_{n\ge 2} \alpha_n^2 w_n \le K^2 \sum_{n\ge 2} n^{2\gamma-1} n w_n \le K^2 \sum_{n\ge 2} n w_n \le  K^2 Q_*.
\end{equation}
Moreover, for $n\ge 2$,
\begin{equation}
\begin{aligned}
{\left| {\left({A}_\mathrm{L}^\alpha {A}_\mathrm{L}^\alpha w\right)}_n \right|}  &= {\left| \alpha_{n+1} {\left(\Alin^\alpha w\right)}_{n+1} - \alpha_{n} {\left(\Alin^\alpha w\right)}_{n} \right|} = {\left| \alpha_{n+2} \alpha_{n+1} w_{n+2} - {\left(\alpha_{n+1}^2 + \alpha_{n+1}\alpha_n\right)} w_{n+1} + \alpha_n^2 w_n \right|} \\
& \le \alpha_{n+2}^2 w_{n+2} + 2\alpha_{n+1}^2 w_{n+1} + \alpha_n^2 w_n.
\end{aligned}
\end{equation}
Therefore, since $0\le \alpha_n^4/ n^2 \le K^4$ and $nw_n \le Q_*$, it holds
\begin{equation}
\sum_{n\ge 2} {\left({A}_\mathrm{L}^\alpha {A}_\mathrm{L}^\alpha w\right)}_n^2 \le 16 \sum_{n\ge 2} \alpha_n^4 w_n^2 \le 16 \sum_{n\ge 2} \frac{\alpha_n^4}{n^2} (n w_n)^2 \le 16K^4 Q_* \sum_{n\ge 2} n w_n \le 16 K^4 Q_*^2.
\end{equation}
In conclusion, ${\left\|{A}_\mathrm{L}^\alpha {A}_\mathrm{L}^\alpha w \right\|} \le 4\sqrt{2}K^2Q_* $. Similar computations can be performed for ${A}_\mathrm{L}^\alpha {A}_\mathrm{L}^\beta w$, ${A}_\mathrm{L}^\beta {A}_\mathrm{L}^\alpha w$ and ${A}_\mathrm{L}^\beta {A}_\mathrm{L}^\beta w$, which leads to the claimed estimates.
\end{proof}

The above estimates allow us in particular to prove Proposition~\ref{prop:uprime}.

\begin{proof}[Proof of Proposition~\ref{prop:uprime}]
Fix a time $T > 0$, a constant $Q_* \in \mathbb{R}_+$ and a non-negative initial condition $u^0 \in \mathcal{Q}$. Suppose that the total quantity of matter of the initial condition satisfies $Q{\left(u^0\right)} \le Q_*$ and denote by $|\cdot|_{\mathcal{C}^0}$ the uniform norm for functions in $\mathcal{C}^0([0,T],\mathbb{R})$, namely $|f|_{\mathcal{C}^0} = \sup_{0 \le t \le T} {\left|f(t) \right|}$. Recall that the total quantity of matter is conserved, so that $Q(u(t)) \le Q_*$. Since $u$ stays non-negative, it holds $0\le  u_n(t)  \le Q_*/n$ for all $t\ge 0$ and $n\ge 1$. In particular, $|u_1|_{\mathcal{C}^0} \le Q_*$. Therefore, for all $0\le t\le T$,
\begin{equation}
{\left\|\frac{du}{dt}(t)\right\|} = {\left\|\Alin^\alpha u(t) + u_1(t) \Alin^\beta u(t)\right\|} \le {\left\|\Alin^\alpha u(t)\right\|} + Q_* {\left\|\Alin^\beta u(t)\right\|} \le (1+Q_*)R(Q_*),
\end{equation}
which concludes the proof of the bound for $du/dt$ in view of Lemma~\ref{lemm:estimates-A-AA}. In particular,
\begin{equation}
\label{eq:estimate-du1}
{\left|\frac{du_1}{dt}\right|}_{\mathcal{C}^0} \le \sup_{0\le t\le T}{\left\|\frac{du}{dt}(t)\right\|} \le (1+Q_*)R(Q_*).
\end{equation}
Then, for all $0\le t\le T$,
\begin{equation}
\begin{aligned}
{\left\|\frac{d^2u}{dt^2}(t)\right\|} &\le {\left\|\Alin^\alpha {\left( \Alin^\alpha + u_1(t)\Alin^\beta\right)} u(t)\right\|} + {\left| u_1 \right|}_{\mathcal{C}^0} {\left\|\Alin^\beta{\left(\Alin^\alpha + u_1(t)\Alin^\beta\right)} u(t)\right\|} + {\left|\frac{du_1}{dt} \right|}_{\mathcal{C}^0} {\left\|\Alin^\beta u(t)\right\|} \\
&\le {\left\|{\left(\Alin^\alpha\right)}^2 u(t)\right\|} + {\left|u_1\right|}_{\mathcal{C}^0} {\left( {\left\|\Alin^\alpha \Alin^\beta u(t)\right\|} +   {\left\|\Alin^\beta \Alin^\alpha u(t)\right\|} \right)} + {\left|u_1\right|}^2_{\mathcal{C}^0} {\left\|{\left(\Alin^\beta\right)}^2 u(t)\right\|} + {\left|\frac{du_1}{dt} \right|}_{\mathcal{C}^0} {\left\|\Alin^\beta u(t)\right\|},
\end{aligned}
\end{equation}
from which we obtain the estimate for $d^2u/dt^2$ in view of Lemma~\ref{lemm:estimates-A-AA} and~\eqref{eq:estimate-du1}.
\end{proof}

\subsection[Proof of the convergence of the splitting]{Proof of Proposition~\ref{prop:splitting}}
\label{sec:appendix-splitting-proof}

We can now write the proof of the convergence of the splitting of the dynamics. The proof can be decomposed in three steps. We first prove the consistency of the splitting for elements of $\mathcal{Q}$ which are bounded in an appropriate norm. We next prove its stability, under the same conditions on elements of $\mathcal{Q}$. We finally conclude to the convergence for arbitrary times using the fact that solutions of~\eqref{problem:P3} are uniformly bounded.

\paragraph{Step 0: Technical results on $\varphi$.} Let us recall that the flow $\varphi_t$ defined in~\eqref{eq:sub-dynamic1} acts only upon the first component of an element $v \in \mathcal{H}$. For $v\in \mathcal{Q}$, denote by
\begin{equation}
\label{eq:def-a-b-c}
a = 2\beta_1,\qquad b(v) = \sum_{n\ge2} \beta_n v_n ,\qquad c(v) = \sum_{n\ge2} \alpha_n v_n + \alpha_2 v_2,
\end{equation}
where $a > 0$ and $b(v),c(v) \ge 0$ are fixed. The dynamics on $t\mapsto \varphi^{(v_2,\cdots)}_t(v_1)$ therefore writes
\begin{equation}
\label{eq:edo-varphi}
\frac{d\varphi^{(v_2,\cdots)}_t(v_1)}{dt} = -a {\left(\varphi^{(v_2,\cdots)}_t(v_1)\right)}^2 - b(v) \varphi^{(v_2,\cdots)}_t(v_1) + c(v).
\end{equation}
In order to prove stability and consistency results on the flow $\varphi_t$, we need the following technical results (the proof of the first one is given at the end of this section).
\begin{lemm}
\label{lemm:estimate-for-varphi}
Fix a time $t\ge 0$ and $Q_* > 0$. Then, there is $\mathcal{B}(Q_*) \in \mathbb{R}_+$ such that, for any $v\in \mathcal{Q}$ with $Q(v) \le Q_*$ and $(v_2,v_3\cdots) \neq (0,0,\cdots)$, it holds~$\varphi^{(v_2,\cdots)}_t(v_1) \ge 0$ for all $t\ge 0$, and
\begin{equation}
\label{eq:splitting-varphi-bound-t}
{\left|\varphi^{(v_2,\cdots)}_t(v_1)\right|},  {\left|\frac{d \varphi^{(v_2,\cdots)}_t(v_1)}{dt}\right|},  {\left|\frac{d^2\varphi^{(v_2,\cdots)}_t(v_1)}{dt^2}\right|}  \le \mathcal{B}(Q_*).
\end{equation}
\end{lemm}

\begin{lemm}
\label{lemm:chi-t-positive}
Consider $v\in\mathcal{Q}$ and suppose that there exists $k\ge 2$ such that $v_k > 0$. Then, for all $t\ge 0$, there exists $\ell \ge 2$ such that $(\chi_t(v))_\ell > 0$.
\end{lemm}
The proof of this result is based on the observation that the subdynamics~\eqref{eq:sub-dynamic2} is a Kolmogorov forward equation of a Markov jump process. The solution therefore remains non-negative. Moreover, since every state is accessible from the state $k$ for this Markov process, the solution of the Kolmogorov forward equation satisfies $u_n(t) > 0$ for all $t > 0$ and all $n\ge 1$ (see~\cite[Chapter 3.2]{norris1998markov}).

\paragraph{Step 1: Consistency.} We prove that, for any $Q_* > 0$, there exists a constant $L_1(Q_*) \in \mathbb{R}_+$ such that, for all $u^0 \in \mathcal{Q}$ with $Q{\left(u^0\right)} \le Q_*$, it holds
\begin{equation}
\label{eq:consistency}
\forall\,0\le \Delta t\le 1,\qquad {\left\|S_{\Delta t}{\left(u^0\right)} - u(\Delta t) \right\|} \le L_1(Q_*) \Delta t^2,
\end{equation}
where $u(t)$ is the solution of~\eqref{problem:P3} at time $t$ with initial condition $u^0$. We first estimate the error between $u_1(\Delta t)$ and $\varphi_{\Delta t}(u_1^0)$, before quantifying the error between $u(\Delta t)$ and $S_{\Delta t}{\left(u^0\right)}$. In the remainder of this part, we fix $Q_* > 0$.  Moreover, recall that the flows $\varphi$ and $\chi$ preserve the non-negativity (see the proofs of Lemmas~\ref{lemm:estimate-for-varphi} and~\ref{lemm:splitting-remains-Q} respectively).\\
\\
\textbf{Step 1.1: Error estimate on $u_1(\Delta t) - \varphi_{\Delta t}{\left(u_1^0\right)}$.} We first show that there is $P_1(Q_*) \in \mathbb{R}_+$ such that, for all $u^0 \in \mathcal{Q}$ with $Q{\left(u^0\right)} \le Q_*$, it holds
\begin{equation}
\label{eq:consistence-u1}
\forall\,0\le \Delta t\le 1,\qquad {\left|u_1(\Delta t) - \varphi_{\Delta t}{\left(u_1^0\right)}\right|} \le P_1(Q_*) \Delta t^2.
\end{equation}
The dynamics on $t\mapsto \varphi_t{\left(u_1^0\right)}$ reads
\begin{equation}
\label{eq:varphi-t-u0}
\frac{d\varphi_t{\left(u_1^0\right)}}{dt} = -a{\left(\varphi_t{\left(u_1^0\right)}\right)}^2 - b(u(0)) \varphi_t{\left(u_1^0\right)} + c(u(0)), \qquad \varphi_0{\left(u_1^0\right)} = u_1^0,
\end{equation}
while the one on $t\mapsto u_1(t)$ reads
\begin{equation}
\frac{du_1}{dt} = -a u_1(t)^2 - b(u(t)) u_1(t) +  c(u(t)),\qquad u_1(0) = u_1^0,
\end{equation}
where $b$ and $c$ are defined in \eqref{eq:def-a-b-c}. Since $t\mapsto u_1(t)$ and $ t\mapsto \varphi_t{\left(u_1^0\right)}$ are twice continuously differentiable (see Proposition~\ref{prop:uprime} and Lemma~\ref{lemm:estimate-for-varphi}), and 
\begin{equation}
\frac{du_1}{dt}(0) = {\left.\frac{d\varphi_t{\left(u_1^0\right)}}{dt}\right|}_{t=0},
\end{equation}
it follows that
\begin{equation}
{\left|u_1(\Delta t) - \varphi_{\Delta t}{\left(u_1^0\right)}\right|} \le \frac{1}{2}\Delta t^2 {\left[ \sup_{0\le\theta\le \Delta t} {\left|\frac{d^2u_1}{dt^2}(\theta)\right|} + \sup_{0\le\theta\le \Delta t} {\left|\frac{d^2\varphi_t{\left(u_1^0\right)}}{dt^2}\right|}\right]}.   
\end{equation}
The second order derivative $d^2\varphi_t{\left(u_1^0\right)}/dt^2$ is uniformly bounded in time by $\mathcal{B}(Q_*)$ (see~\eqref{eq:splitting-varphi-bound-t}). Moreover, in view of Proposition~\ref{prop:uprime}, $d^2u_1/dt^2$ is also uniformly bounded in time, by a constant which depends on $Q_*$. This leads to~\eqref{eq:consistence-u1}.\\
\\
\textbf{Step 1.2: Error estimates on $\Pi(S_{\Delta t}(u^0) - u(\Delta t))$.} We prove that there is $P_2(Q_*) \in \mathbb{R}_+$ such that, for all $u^0 \in \mathcal{Q}$ with $Q{\left(u^0\right)} \le Q_*$, it holds
\begin{equation}
\label{eq:consistence-Pi-u}
\forall\, 0\le \Delta t\le 1,\qquad {\left\| \Pi{\left(u(\Delta t) - S_{\Delta t}{\left(u^0\right)}\right)} \right\|} \le P_2(Q_*) \Delta t^2.
\end{equation}
Let us first reinterpret $\Pi S_{\Delta t}$ as the flow of some time continuous dynamics. We rewrite to this end~\eqref{problem:sub-linear} as
\begin{equation}
\frac{d\chi_t{\left(u^0\right)}}{dt} = {\left(\Pi\Alin^\alpha + \varphi_{\Delta t}{\left(u_1^0\right)}\Pi\Alin^\beta \right)} \chi_t{\left(u^0\right)}, \qquad \chi_0{\left(u^0\right)} = {\left(\varphi_{\Delta t}{\left(u_1^0\right)}, u_2^0, \cdots\right)}.
\end{equation}
Consider $\widetilde{S}_t = \chi_t \circ \varphi_{\Delta t}$ and note that $S_{\Delta t} = \widetilde{S}_{\Delta t}$. Moreover, $w = \Pi{\left(u-\widetilde{S}_t{\left(u^0\right)}\right)}$ is solution of
\begin{equation}
\frac{dw}{dt}(t) = A^\Pi{\left(\varphi_{\Delta t}{\left(u_1^0\right)}\right)} w(t) + {\left(u_1(t)-\varphi_{\Delta t}{\left(u_1^0\right)}\right)}\Pi\Alin^\beta u(t), \qquad w(0) = 0.
\end{equation}
Using Proposition~\ref{prop:splitting-semigroup}, since $t\mapsto u_1(t) - \varphi_{\Delta t}{\left(u_1^0\right)}$ and $t\mapsto \Alin^\beta u(t)$ are continuously differentiable, we can write
\begin{equation}
w(\Delta t) = \int_0^{\Delta t} T_{\varphi_{\Delta t}{\left(u_1^0\right)}}(\Delta t-s){\left[u_1(s) - \varphi_{\Delta t}{\left(u_1^0\right)}\right]}\Pi\Alin^\beta u(s)\, ds, 
\end{equation}
so that
\begin{equation}
{\left\| w(\Delta t) \right\|} \le \Delta t \sup_{0\le s \le \Delta t} {\left\| T_{\varphi_{\Delta t}{\left(u_1^0\right)}}(s) \right\|}  \sup_{0\le s\le \Delta t} {\left\{ {\left|u_1(s) - \varphi_{\Delta t}{\left(u_1^0\right)}\right|} {\left\| \Pi \Alin^\beta u(s) \right\|} \right\}}.
\end{equation}
Since $\varphi_{t}{\left(u_1^0\right)}$ is bounded by $\mathcal{B}(Q_*)$ (see~\eqref{eq:splitting-varphi-bound-t}), in view of Proposition~\ref{prop:LCD}, it holds, for any $0\le s\le \Delta t$,
${\left\| T_{\varphi_{\Delta t}{\left(u_1^0\right)}} (s) \right\|}  \le \exp{\left(\lambda_{\mathcal{B}(Q_*)} \Delta t\right)}$.
Moreover, in view of Proposition~\ref{prop:uprime} and Lemma~\ref{lemm:estimate-for-varphi}, and since $u_1(0) = u_1^0$,
\begin{equation}
{\left|u_1(s) - \varphi_{\Delta t}{\left(u_1^0\right)}\right|} \le \Delta t {\left( {\left|\frac{du_1}{dt}\right|}_{\mathcal{C}^0([0,\Delta t])} + {\left|\frac{d\varphi_t{\left(u_1^0\right)}}{dt}\right|}_{\mathcal{C}^0([0,\Delta t])}  \right)} \le \Delta t{\left( R(Q_*) + \mathcal{B}(Q_*)\right)}.
\end{equation}
Finally, in view of Lemma~\ref{lemm:estimates-A-AA}, since $Q(u(s)) = Q{\left(u^0\right)} \le Q_*$ for any $0\le s\le \Delta t$ (see Theorem~\ref{thm:long-time}), we obtain ${\left\| \Pi \Alin^\beta u(s) \right\|} \le R(Q_*)$. Then, ${\left\| \Pi{\left(u(\Delta t) - S_{\Delta t}{\left(u^0\right)}\right)} \right\|} \le \Delta t^2 \exp{\left(\lambda_{\mathcal{B}(Q_*)} \Delta t\right)}\big[ R(Q_*) + \mathcal{B}(Q_*) \big]R(Q_*)$, which leads to~\eqref{eq:consistence-Pi-u}. \\
\\
\textbf{Step 1.3: The splitting is consistent.} Consider now $u^0 \in \mathcal{Q}$ with $Q{\left(u^0\right)} \le Q_*$. We first note that ${\left\|S_{\Delta t}{\left(u^0\right)} - u(\Delta t)\right\|}^2 = {\left\|{\left(u_1(\Delta t) - \varphi_{\Delta t}{\left(u^0\right)},0,\cdots\right)}\right\|}^2 + {\left\| \Pi{\left(u(\Delta t) - S_{\Delta t}{\left(u^0\right)}\right)} \right\|}^2$. The estimate~\eqref{eq:consistency} then follows from~\eqref{eq:consistence-u1} and~\eqref{eq:consistence-Pi-u}.

\paragraph{Step 2: Stability.} We prove that, for any $Q_* > 0$, there exists $L_2(Q_*) \in \mathbb{R}_+$, such that for all $u,v \in \mathcal{Q}$ with $Q(u), Q(v) \le Q_*$, it holds
\begin{equation}
\label{eq:stability}
\forall\, 0\le \Delta t \le 1, \qquad {\left\| S_{\Delta t}(u) - S_{\Delta t}(v) \right\|} \le \exp(L_2(Q_*)\Delta t) {\left\| u - v \right\|} + 2 L_1(Q_*)  \Delta t^2 .
\end{equation}
Fix $Q_* > 0$ and consider $u,v \in \mathcal{Q}$ with $Q(u), Q(v) \le Q_*$. Denote by $\widetilde{u},\widetilde{v}$ the solutions of~\eqref{problem:P3} with initial conditions $u, v$ respectively. Then, ${\left\| S_{\Delta t}(u) - S_{\Delta t}(v) \right\|} \le  {\left\| S_{\Delta t}(u) - \widetilde{u}(\Delta t) \right\|} +  {\left\| \widetilde{u}(\Delta t) - \widetilde{v}(\Delta t) \right\|} +  {\left\| S_{\Delta t}(u) - \widetilde{v}(\Delta t) \right\|}$, where ${\left\| S_{\Delta t}(u) - \widetilde{u}(\Delta t) \right\|} + {\left\| S_{\Delta t}(u) - \widetilde{v}(\Delta t) \right\|} \le 2 L_1(Q_*)\Delta t^2$ in view of~\eqref{eq:consistency}. Moreover, in view of~\eqref{eq:estimate-d-u-v-2}, it holds
\begin{equation}
\frac{d\|\widetilde{u}-\widetilde{v}\|^2}{dt} \le (\lambda_{Q_*} + R(Q_*)) \|\widetilde{u}-\widetilde{v}\|^2.
\end{equation}
Therefore, using a Gronwall inequality, there exists $L_2(Q_*) \in \mathbb{R}_+$ such that ${\left\| \widetilde{u}(\Delta t) - \widetilde{v}(\Delta t) \right\|} \le \exp(L_2(Q_*)\Delta t)\|u - v\|$. The estimate~\eqref{eq:stability} follows by combining the latter inequality with the consistency estimates~\eqref{eq:consistency}.

\paragraph{Step 3: Convergence.} The convergence of the splitting as $\Delta t \to 0$ classically follows from the stability and consistency estimates obtained in Steps~1 and~2. We however first need to make sure that $Q{\left(u^n\right)} \le Q_*$ in order to apply~\eqref{eq:consistency} and~\eqref{eq:stability} for a well-chosen $Q_*$. The proof proceeds by induction.

Fix a time $\tau > 0$, an initial condition $u^0 \in \mathcal{Q}$ and let $Q_0 = Q{{\left(u^0\right)}}$. Our aim is to prove that there exist $\Delta t_* > 0$ and~$L(\tau,u^0) \in \mathbb{R}_+$ such that
\begin{equation}
\label{eq:convergence}
\forall\,0 < \Delta t \le \Delta t_*,\qquad \forall\, 0\le n\le \frac{\tau}{\Delta t}, \qquad {\left\|u(n\Delta t) - u^n \right\|} \le L(\tau,u^0) \Delta t.
\end{equation}
In fact, as made precise below, the constant $L(\tau,u^0)$ depends on~$u^0$ only through $Q_0$ and $M_* = 2 \sup_{0\le t \le \tau} {\left\| u(t) \right\|}$. More precisely, consider $Q_* = 2Q_0\exp{\left(2K M_* \tau\right)}$ (with $K$ as in Assumption~\ref{hypo:alpha-beta-gamma}), as well as the constant $\mathcal{K}(\tau,Q_*) = 3L_1(Q_*)\tau \exp(L_2(Q_*)\tau)$ (where the prefactors $L_i(Q_*)$ are the ones appearing in~\eqref{eq:consistency} and~\eqref{eq:stability}), and the time step 
\begin{equation}
\label{eq:def-detlat-star}
\Delta t_* = \min{\left(1, \frac{M_*}{2\mathcal{K}(\tau,Q_*)}, \frac{\ln(2)}{2K\mathcal{B}(Q_*) \tau} \right)},
\end{equation}
where $\mathcal{B}(Q_*)$ is the constant appearing in Lemma~\ref{lemm:estimate-for-varphi}. Fix $0 < \Delta t \le \Delta t_*$. We prove by induction that, for $0 \le n \le \tau/\Delta t$, it holds
\begin{equation}
\label{eq:splitting-induction}
\forall\, 0\le k \le n, \quad Q{\left(u^k\right)} \le Q_0\exp{\left(2K(M_*+\Delta t \mathcal{B}(Q_*))k\Delta t\right)}, \quad {\left\|u^k\right\|} \le M_*,\quad {\left\|u(k\Delta t) - u^k \right\|} \le \mathcal{K}(\tau,Q_*) \Delta t.
\end{equation}
The induction basis is clear since $u^0 = u(0)$. Assume now that~\eqref{eq:splitting-induction} holds for some integer $0 \le n \le \tau/\Delta t$ such that $n+1 \le \tau/\Delta t$. First, in view of Lemma~\ref{lemm:splitting-remains-Q}, 
$Q{\left(S_{\Delta t}{\big(u^n\big)}\right)} \le Q{\left(u^n\right)}\exp{\left(2K \varphi_{\Delta t}{\left(u_1^n\right)} \Delta t\right)}$.
Moreover, in view of Lemma~\ref{lemm:estimate-for-varphi}, we also have $\varphi_{\Delta t}{\left(u^n_1\right)} \le {\left|u_1^n\right|} + \Delta t \mathcal{B}(Q_*)$ and ${\left|u_1^n \right|} \le {\left\| u^n \right\|} \le M_*$. Then, by the induction hypothesis, $Q{\left(u^{n+1}\right)} \le Q_0\exp{\left(2K(M_*+\Delta t \mathcal{B}(Q_*))(n+1)\Delta t\right)}$. Therefore, $Q{\left(u^{n+1}\right)} \le Q_0 \exp{\left(2K M_* \tau\right)}\exp(2K\mathcal{B}(Q_*)\tau\Delta t) \le Q_*$, where we used the fact that $\Delta t \le \ln(2){\left(2K\mathcal{B}(Q_*) \tau\right)}^{-1}$ so that $\exp(2K\mathcal{B}(Q_*)\tau\Delta t) \le 2$, as well as the fact that $Q_0\exp{\left(2K M_* \tau\right)} = Q_*/2$. We are then in position to prove the two other inequalities in~\eqref{eq:splitting-induction}. Note that $Q{\left(u(t)\right)} = Q_0 \le Q_*$ for all~$t\ge 0$. Therefore, in view of the estimates~\eqref{eq:consistency} and~\eqref{eq:stability}, for all~$0 < \Delta t \le \Delta t_*$ and all~$0\le n \le \tau/\Delta t$, it holds
\begin{equation}
\begin{aligned}
{\left\|u((n+1)\Delta t) - u^{n+1}\right\|} &\le {\left\|u((n+1)\Delta t) - S_{\Delta t}\big(u(n\Delta t)\big)\right\|} + {\left\|S_{\Delta t}\big(u(n\Delta t)\big) - u^{n+1}\right\|} \\
&\le 3L_1(Q_*)\Delta t^2 + \exp(L_2(Q_*)\Delta t) {\left\| u(n\Delta t) - u^{n} \right\|}. \end{aligned}
\end{equation}
Since $u(0) = u^0$, we obtain by recursion
\begin{equation}
{\left\|u((n+1)\Delta t) - u^{n+1}\right\|} \le 3L_1(Q_*) \Delta t^2 \sum_{k=0}^n \exp(L_2(Q_*) k \Delta t) \le 3L_1(Q_*) \tau \exp(L_2(Q_*)\tau) \Delta t = \mathcal{K}(\tau,Q_*) \Delta t,
\end{equation}
from which the last inequality in~\eqref{eq:splitting-induction} follows for $n+1$. Finally, using a reverse triangle inequality and~\eqref{eq:def-detlat-star}, it holds
${\left\| u^{n+1} \right\|} \le \mathcal{K}(\tau,Q_*)\Delta t + \frac{1}{2}M_* \le M_*$,
which concludes the proof.\qed

\medskip

Let us conclude this section by providing the proof of Lemma~\ref{lemm:estimate-for-varphi}.
\begin{proof}[Proof of Lemma~\ref{lemm:estimate-for-varphi}]
The unique solution of~\eqref{eq:edo-varphi} reads (see~\cite{Terrier})
\begin{equation}
\label{eq:sol-u1S}
\varphi^{(v_2,\cdots)}_t(v_1) = \frac{v_1[r_+(v)-r_-(v)\exp(-\delta(v) t)] - r_+(v) r_-(v)(1-\exp(-\delta(v) t))}{v_1[1-\exp(-\delta(v) t)] + r_+(v)\exp(-\delta(v) t) - r_-(v)},
\end{equation}
with $r_{+}(v) = -\frac{b(v)-\sqrt{b(v)^2+4ac(v)}}{2a} > 0$ and $r_{-}(v) = -\frac{b(v)+\sqrt{b(v)^2+4ac(v)}}{2a} < 0$. In view of the definition of $r_+(v)$ and $r_-(v)$, and since $\exp(-\delta(v) t) \le 1$, the terms $r_+(v)-r_-(v)\exp(-\delta(v) t)$ and $ r_+(v)\exp(-\delta(v) t) - r_-(v) $ are positive, while $ - r_+(v) r_-(v)(1-\exp(-\delta(v) t)) $ and $ 1-\exp(-\delta(v) t) $ are non-negative. Therefore, if $v_1 > 0$, the solution $\varphi^{(v_2,\cdots)}_t(v_1)$ remains positive for all times $t\ge 0$.

We next prove the estimates~\eqref{eq:splitting-varphi-bound-t}. Since $Q(v) \le Q_*$, we have in particular ${\left\|\Alin^\alpha v\right\|} \le R(Q_*)$ in view of Lemma~\ref{lemm:estimates-A-AA}. Therefore $c(v)= {\left(\Alin^\alpha v\right)}_1 \le R(Q_*)$. Similarly, $b(v) = {\left(\Alin^\beta v\right)}_1 \le R(Q_*)$, so that
$r_+(v)^2 \le \left( b(v)^2 + 2 ac(v) \right)/(2a^2) \le \left( R(Q_*)^2+4\beta_1 R(Q_*) \right)/(8\beta_1^2)$.
In view of~\eqref{eq:sol-u1S}, it holds
\begin{equation}
\begin{aligned}
  {\left|\varphi^{(v_2,\cdots)}_t(v_1)\right|} &\le \frac{v_1(r_+(v) - r_-(v)) - r_-(v) r_+(v)}{-r_-(v)} \le 2v_1 + r_+(v)
  \le \mathcal{R}(Q_*).
\end{aligned}
\end{equation}
Moreover, in view of~\eqref{eq:edo-varphi}, $d\varphi^{(v_2,\cdots)}_t(v_1)/dt$ is uniformly bounded in time by $2\beta_1 \mathcal{R}(Q_*)^2 + R(Q_*)\mathcal{R}(Q_*) + R(Q_*)$. We finally note that
\begin{equation}
\frac{d^2\varphi_t^{(v_2,\cdots)}{\left(v_1\right)}}{dt^2}(t) = -{\left(2a^2 \varphi^{(v_2,\cdots)}_t{\left(v_1\right)} + b(v)\right)} \frac{d\varphi^{(v_2,\cdots)}_t{\left(v_1\right)}}{dt},
\end{equation}
is also uniformly bounded in time since $\varphi^{(v_2,\cdots)}_t(v_1)$ and $d\varphi^{(v_2,\cdots)}_t(v_1)/dt$ are uniformly bounded. This shows that the estimates~\eqref{eq:splitting-varphi-bound-t} hold true.
\end{proof}

\section[Proofs for the decay estimates of the solution to the diffusion equation]{Proofs for the decay estimates of the solution of~\eqref{problem:Cauchy-diffusion}}
\label{sec:appendix-proofs-FPCD}

This section is organized as follows. We first prove Theorem~\ref{thm:regularity} in Appendix~\ref{sec:appendix-decay}, with some technical estimates postponed to Appendix~\ref{sec:appendix-decay-details}. We finally prove Theorems~\ref{thm:Diff2FP} and~\ref{thm:Diff2CD} in Appendix~\ref{sec:appendix-proofs-CD-FP-Diff}.

\subsection[Proof of the Theorem]{Proof of Theorem~\ref{thm:regularity}}
\label{sec:appendix-decay}
 
The existence, uniqueness and regularity of the solution is made precise in~\cite{Terrier}. These properties follow by an application of Hörmander's Theorem in a stochastic version~\cite{hairer2011hormander}. 

Let us now turn to the decay estimates. We first prove~\eqref{eq:decay-estimates} for $n=1$ by introducing the stochastic process associated with~\eqref{problem:Cauchy-diffusion} (Steps~1-2). In Step~3 we generalize the previous steps and rely on results proved in Appendix~\ref{sec:appendix-decay-details} in order to give a proof for higher order derivatives. \\
\\
\textbf{Step 1: Reformulation as a diffusion with additive noise.} We introduce the stochastic process $(X_t)_{t\ge 0}$ defined as $X_0 = \varx$ and $dX_t = \sigma(X_t)\,dW_t$, where $(W_t)_{t\ge 0}$ is a standard Brownian motion. Using the Feynman-Kac representation formula~\cite{lebris2008existence}, the solution of~\eqref{problem:Cauchy-diffusion} can be written as $\fun(\vart,\varx) = \mathbb{E}{\left[ \Cdiff_0(X_\vart) \middle| X_0 = \varx \right]} := \mathbb{E}^\varx{\left[ \Cdiff_0(X_\vart) \right]}$. Note that the boundedness of $\Cdiff_0$ immediately gives the boundedness of $u$. We next use a Lamperti transform~\cite{pages2006lamperti} on the process $X$. Define
\begin{equation}
\varphi(\varx) = \int_0^\varx \frac{1}{\sigma(s)}\, ds.
\label{eq:phi-lamperti}
\end{equation}
The function $\varphi$ is a well defined smooth function since $\sigma$ is a positive smooth function. We then introduce the stochastic process $Y_t = \varphi(X_t)$. Using Itô's formula,
\begin{equation}
dY_t = -\frac{1}{2}\sigma^\prime{\left(\varphi^{-1}(Y_t)\right)}dt + dW_t, \qquad Y_0 = \varphi(\varx) = y.
\label{eq:Y-process}
\end{equation} 
Defining $v_0 = \Cdiff_0 \circ \varphi^{-1}$, and the function
\begin{equation}
\forall\, (t,y) \in \mathbb{R}_+ \times \mathbb{R},\qquad v(\vart,\vary) := \mathbb{E}^{y}{\left[ v_0(Y_t)\right]},
\label{eq:def-v-esperance}
\end{equation}
the values of $u$ are obtained from $u(t,x) = v(t,\varphi(x))$.\\
\\
\textbf{Step 2: Relating the first derivative of $u$ with the flow.} Introduce $\Psi = -\frac{1}{2}\sigma^\prime \circ \varphi^{-1}$, which is a  smooth bounded function with bounded derivatives (see Appendix~\ref{sec:appendix-Psi-0}), and the tangent process $\eta$ of $Y$ (\textit{i.e.} the derivative of the flow with respect to the initial condition~\cite[Chapter V, Theorem 39]{protter2013stochastic}):
\begin{equation}
d\eta_t = \Psi^\prime(Y_t)\eta_t\, dt, \qquad \eta_0 = 1.
\label{eq:tangent-process}
\end{equation}
Then (see~\cite[Chapter 1.3]{cerrai2001second}),
\begin{equation}
\frac{\partial v}{\partial \vary}(\vart,\vary) = \mathbb{E}^{y}{\left[ v^\prime_0(Y_t)\eta_t\right]}.
\label{eq:derivee-esperance}
\end{equation}
In order to bound the derivative $\partial v/\partial y$, we first notice that $\eta$ is simply the solution of an ODE with a continuous stochastic coefficient:
\begin{equation}
\eta_t = \eta_0 \exp{\left[\int_0^t \Psi^\prime(Y_s)\,ds\right]}.
\label{eq:solution-eta}
\end{equation}
Since $\Psi^\prime$ is bounded, there exists $M_\Psi \in \mathbb{R}_+^*$ such that $0 \le \eta_t \le \eta_0 \exp(M_\Psi t)$ for all $ t\ge 0$. Moreover, $v_0^\prime$ is bounded in view of Lemma~\ref{lemm:estimate-v-0} in Appendix~\ref{sec:appendix-v-0}. Therefore, there exists $L\ge 0$ such that, for all $t\ge 0$,
\begin{equation}
{\left| \frac{\partial v}{\partial y}(t,y) \right|} \le L \exp{\left(M_{\Psi} t\right)}.
\end{equation}
The estimate~\eqref{eq:decay-estimates} for $n=1$ is finally obtained by noting that $\partial_\varx u(\vart,\varx) = \partial_\varx \big[v(\vart,\varphi(\varx))\big] = \partial_\vary v(\vart,\varphi(\varx)) \varphi^\prime(\varx)$, so that
\begin{equation}
\forall\, (t,x) \in \mathbb{R}_+ \times \mathbb{R}, \qquad  {\left| \frac{\partial u}{\partial \varx}(\vart,\varx) \right|} \le \frac{L \exp{\left(M_{\Psi} t\right)}}{\sigma(\varx)}.
\end{equation}
The result then follows from Assumption~\ref{hypo:sigma-diffusion}.\\
\\
\textbf{Step 3: Generalizing to higher order derivatives.} For the remainder of the proof, let us introduce the tangent process $\eta^{(n)}$  of order $n$ of $Y_t$, recursively defined as the tangent process of $\eta^{(n-1)}$ (see Lemma~\ref{lemm:eta-n}). We also have the following results, proved in Appendix~\ref{sec:appendix-decay-details}: 
\begin{enumerate}
\item For all $n\ge 1$, there exists a non-negative function $\widetilde{K}_n \in \mathcal{C}^0(\mathbb{R}_+)$ such that (see Appendix~\ref{sec:appendix-eta-0})
\begin{equation}
\label{eq:estimate-eta-n}
\forall\, t\ge 0,\qquad 0 \le \eta^{(n)}_t \le \widetilde{K}_n(t):
\end{equation}
\item For all $n\ge 1$, the derivative of order $n$ of $v_0$ is bounded (see Appendix~\ref{sec:appendix-v-0});
\item For all $n \ge 1$, it holds, for $|x| \to +\infty$ (see Appendix~\ref{sec:appendix-phi-0}), $\varphi^{(n)}(\varx) = \mathrm{O}{\left(|x|^{-\gamma/2-n+1}\right)}$. 
\end{enumerate}
We then use the Faà di Bruno's formula~\cite{johnson2002curious} in order to write the higher order derivatives of $v$. Recall that, for $f,g \in \mathcal{C}^\infty(\mathbb{R})$ and $n\ge 1$,
\begin{equation}
\label{eq:faadibruno}
\frac{d^n}{dx^n} \big( f(g(x)) \big) = \sum_{k=1}^n f^{(k)}(g(x)) B_{n,k}{\left(g'(x),g''(x),\dots,g^{(n-k+1)}(x)\right)},
\end{equation}
where $B_{n,k}$ are the Bell polynomials~\cite{bell1927partition} (see Section~\ref{sec:appendix-bell}). Then, the Faà di Bruno's formula applied to~\eqref{eq:def-v-esperance} together with the results of~\cite{cerrai2001second} leads to the following equality: for all $n\ge 1$,
\begin{equation}
\label{eq:faadi-v}
\forall\, (t,y) \in \mathbb{R}_+ \times \mathbb{R}, \qquad \frac{\partial^n v}{\partial \vary^n}(\vart,\vary) = \mathbb{E}^\vary{\left[ \sum_{k=1}^n v_0^{(k)}(Y_t) B_{n,k}{\left(\eta_t,\eta_t^{(2)},\dots,\eta^{(n-k+1)}_t\right)} \right]}.
\end{equation}
Using~\eqref{eq:estimate-eta-n} and Lemma~\ref{lemm:estimate-v-0} in Section~\ref{sec:appendix-v-0}, there exists, for all $n\ge 1$, a non-negative function $M_n \in \mathcal{C}^0(\mathbb{R}_+)$ such that
\begin{equation}
\label{eq:majoration-v}
\forall\, n\ge 1,\qquad \forall\, (t,y) \in \mathbb{R}_+ \times \mathbb{R}, \qquad  {\left| \frac{\partial^n v}{\partial \vary^n}(\vart,\vary) \right|} \le M_n(t).
\end{equation}
We then use once again the Faa di Bruno's formula to compute the $n$-th partial derivative of $u$:
\begin{equation}
\label{eq:faadi-u}
\frac{\partial^n u}{\partial \varx^n}(\vart,\varx) = \sum_{k=1}^n \frac{\partial^{k} v}{\partial \vary^k}(\vart,\varphi(\varx)) B_{n,k}{\left(\varphi^\prime(\varx),\varphi^{(2)}(\varx),\dots,\varphi^{(n-k+1)}(\varx)\right)},
\end{equation}
so that, with the estimate~\eqref{eq:majoration-v}, one gets:
\begin{equation}
{\left| \frac{\partial^n u}{\partial \varx^n}(\vart,\varx) \right|} \le \sum_{k=1}^n M_k(t) B_{n,k}{\left({\left|\varphi^\prime(\varx)\right|},{\left|\varphi^{(2)}(\varx)\right|},\dots,{\left|\varphi^{(n-k+1)}(\varx)\right|}\right)}.
\end{equation}
Moreover, in view of Lemma~\ref{lemm:Bnk}, it holds
\begin{equation}
\label{eq:estimate-Bnk-varphi}
\forall\, 1 \le k \le n, \qquad B_{n,k}{\left({\left|\varphi^\prime(\varx)\right|},{\left|\varphi^{(2)}(\varx)\right|},\dots,{\left|\varphi^{(n-k+1)}(\varx)\right|}\right)} = \mathrm{O}{\left(|x|^{-n + k - k\gamma/2}\right)}.
\end{equation}
Therefore, we obtain that there exists a non-negative function $K_n \in \mathcal{C}^0(\mathbb{R}_+)$ such that ${\left|\partial_x^n u(\vart,\varx) \right|} \le K_n(t) |x|^{-n\gamma/2}$, which concludes the proof of Theorem~\ref{thm:regularity}.

\subsection{Some technical results on $\varphi$, $\Psi$, $\eta$ and their derivatives}
\label{sec:appendix-decay-details}

We gather in this section all the technical results used in the proof of Theorem~\ref{thm:regularity}. We will repeatedly use the Faà di Bruno's formula~\eqref{eq:faadibruno} and Bell Polynomials.

\subsubsection{Bell polynomials}
\label{sec:appendix-bell}

Bell Polynomials~\cite{bell1927partition} are defined as follows: for any $1\le k\le n$,
\begin{equation}
B_{n,k}(x_1,x_2,\dots,x_{n-k+1}) = \!\!\! \sum_{(j_1,\cdots,j_{n-k+1}) \in \mathcal{B}_{n,k}} \!\!\! \frac{n!}{j_1!j_2!\cdots j_{n-k+1}!}
{\left(\frac{x_1}{1!}\right)}^{j_1}\cdots{\left(\frac{x_{n-k+1}}{(n-k+1)!}\right)}^{j_{n-k+1}},
\end{equation}
where $\mathcal{B}_{n,k}$ is the set of all sequences $(j_1,j_2,\cdots,j_{n-k+1})$ of non-negative integers such that
\begin{equation}
\label{eq:indices-bell-polynomials}
{\left\{
\begin{aligned}
&j_1 + j_2 + \cdots + j_{n-k+1} = k,\\
&j_1 + 2 j_2 + 3 j_3 + \cdots + (n-k+1)j_{n-k+1} = n.
\end{aligned}
\right.}
\end{equation}

\subsubsection{Some estimates on $\varphi$}
\label{sec:appendix-phi-0}

\begin{lemm}
\label{lemm:estimates-varphi-0}
The function $\varphi$ defined in~\eqref{eq:phi-lamperti} is smooth and its derivatives satisfy:
\begin{equation}
\forall\, n\ge 1, \qquad  \exists\, K_n \ge 0, \qquad \forall\, |x| \ge 1, \qquad {\left|\varphi^{(n)}(x)\right|} \le \frac{K_n}{|x|^{n-1+\gamma/2}}.
\end{equation}
\end{lemm}
\begin{proof}
By definition, since $\sigma$ is a smooth positive function with bounded derivatives, $\varphi$ is also a smooth function and the estimate holds true for $n=1$. Moreover, in view of Assumption~\ref{hypo:sigma-diffusion}, there exists $L_1, L_2 \in \mathbb{R}_+^*$ such that
\begin{equation}
\forall\, |x|\ge 1, \qquad L_1 |x|^{\gamma/2} \le \sigma(x) \le L_2 |x|^{\gamma/2}.
\end{equation}
Then, using the Faà-di-Bruno's formula, for all $n\ge 1$, it holds, since $\varphi^\prime(x) = 1/\sigma(x)$,
\begin{equation}
\varphi^{(n+1)}(x) = \sum_{k=1}^n \frac{(-1)^k k!}{\sigma(x)^{k+1}}B_{n,k} {\left( \sigma^\prime(x),\cdots,\sigma^{(n-k+1)}(x)\right)}.
\end{equation}
In view of the definition of Bell polynomials (see Appendix~\ref{sec:appendix-bell}) and Assumption~\ref{hypo:sigma-diffusion}, we obtain, for $|x|\ge 1$,
\begin{equation}
\begin{aligned}
&{\left|B_{n,k}{\left( \sigma^\prime(x),\cdots,\sigma^{(n-k+1)}(x)\right)} \right|} \\
&\qquad \le  \!\!\!\!\!  \mathlarger{\mathlarger{\sum}}_{\substack{(j_1,\cdots,j_{n-k+1}) \in \mathcal{B}_{n,k}}} \!\!   \frac{n!S_1^{j_1}\cdots S_{n-k+1}^{j_{n-k+1}}}{j_1!j_2!\cdots j_{n-k+1}!}
{\left(\frac{|x|^{\gamma/2-1}}{1!}\right)}^{j_1}\cdots{\left(\frac{|x|^{\gamma/2-n+k-1}}{(n-k+1)!}\right)}^{j_{n-k+1}}.
\end{aligned}
\end{equation}
Noting that $j_1 + \cdots + j_{n-k+1} = k$ and $j_1 + 2j_2 + \cdots + (n-k+1)j_{n-k+1} = n$, we obtain that there is a constant~$R_{n,k} \in \mathbb{R}_+$ such that
\begin{equation}
{\left| \frac{(-1)^k k!}{\sigma(x)^{k+1}}B_{n,k} {\left( \sigma^\prime(x),\cdots,\sigma^{(n-k+1)}(x)\right)} \right|} \le \frac{R_{n,k}}{|x|^{n+k\gamma/2}}.
\end{equation}
Since $k\ge 1$, it finally holds
\begin{equation}
{\left|\varphi^{(n+1)}(x)\right|} \le \sum_{k=1}^{n} \frac{R_{n,k}}{|x|^{n+k\gamma/2}} = \mathrm{O}{\left(|x|^{-n-\gamma/2}\right)} ,
\end{equation}
which concludes the proof.
\end{proof}

\begin{lemm}
\label{lemm:Bnk}
For all $1\le k\le n$, it holds $B_{n,k}{\left({\left|\varphi^\prime(\varx)\right|},{\left|\varphi^{(2)}(\varx)\right|},\dots,{\left|\varphi^{(n-k+1)}(\varx)\right|}\right)} = \mathrm{O}{\left(|x|^{-n + k - k\gamma/2}\right)}$.
\end{lemm}

\begin{proof}
In view of Lemma~\ref{lemm:estimates-varphi-0}, there exists $\widetilde{S}_{n,k} \in \mathbb{R}_+$ such that
\begin{equation}
B_{n,k}{\left({\left|\varphi^\prime(\varx)\right|},\dots,{\left|\varphi^{(n-k+1)}(\varx)\right|}\right)} \le \widetilde{S}_{n,k} \sum_{(j_1,\cdots,j_{n-k+1}) \in \mathcal{B}_{n,k}} 
{\left(|x|^{-\gamma/2}\right)}^{j_1}\cdots{\left(|x|^{\gamma/2-n+k}\right)}^{j_{n-k+1}} 
\end{equation}
Note that
\[
{\left(|x|^{-\gamma/2}\right)}^{j_1}{\left(|x|^{-\gamma/2-1}\right)}^{j_2}\cdots{\left(|x|^{\gamma/2-n+k}\right)}^{j_{n-k+1}} = |x|^{(1-\gamma/2)(j_1+\cdots+j_{n-k+1})} |x|^{-(j_1+2j_2+\cdots+(n-k+1)j_{n-k+1})} ,
\]
so that, in view of Appendix~\ref{eq:indices-bell-polynomials},
\begin{equation}
0 \le B_{n,k}{\left({\left|\varphi^\prime(\varx)\right|},{\left|\varphi^{(2)}(\varx)\right|},\dots,{\left|\varphi^{(n-k+1)}(\varx)\right|}\right)}  \le \widetilde{S}_{n,k} \sum_{(j_1,\cdots,j_{n-k+1}) \in \mathcal{B}_{n,k}} |x|^{-k\gamma/2 -n +k},
\end{equation}
which gives the claimed estimate. 
\end{proof}

\subsubsection{On the derivatives of $v_0$}
\label{sec:appendix-v-0}

The following technical result holds true in the framework of Theorem~\ref{thm:regularity} in view of Assumption~\ref{hypo:C0}.
\begin{lemm}
\label{lemm:estimate-v-0}
The function $v_0$ is smooth and its derivatives are uniformly bounded on $\mathbb{R}$.
\end{lemm}
\begin{proof}
By definition, since $\sigma$ is a positive continuous function, $\varphi$ is an increasing continuous function, and is therefore invertible. Moreover, since $\varphi^\prime = 1/\sigma$ is positive, the inverse function $\varphi^{-1}$ is also differentiable and its first derivative reads ${\left(\varphi^{-1}\right)}^\prime(y) = \sigma(\varphi^{-1}(y))$. In fact $\varphi^{-1}$ is smooth and its $n$-th order derivative is~${\left(\varphi^{-1}\right)}^{(n)} = R_n{\left(\sigma,\sigma^\prime,\cdots,\sigma^{(n-1)}\right)}\circ \varphi^{-1}$, where $R_n$ is a polynomial related to the Bell polynomials. Since $\Cdiff_0$ is smooth, this proves that $v_0 = \Cdiff_0 \circ \varphi^{-1}$ is smooth.

Next, in view of the definition of $v_0$, it holds $v_0^\prime(y) = \Cdiff_0^\prime(\varphi^{-1}(y))/\varphi^\prime(\varphi^{-1}(y)) = {\left(\Cdiff^\prime \sigma\right)} {\left(\varphi^{-1}(y)\right)}$. Note that, in view of Assumption~\ref{hypo:C0}, $\Cdiff_0^\prime \sigma$ is bounded. Therefore, there exists $R_1 \ge 0$ such that $|v_0^\prime(y)| \le R_1$ for all $y \in \mathbb{R}$. A similar argument is used for higher order derivatives. We first prove that the derivative of order $n$ of $v_0$ reads $v_0^{(n)} = P_n \circ \varphi^{-1}$, with
\begin{equation}
\label{eq:v_0-Pn}
P_n = \Cdiff_0^{(n)}\sigma^n + \sum_{k = 1}^{n-1} \Cdiff_0^{(k)} \sigma^k {\left[ \sum_{j=1}^k \sum_{{\left(\ell^k_1,\cdots,\ell^k_j\right)}\in \mathcal{L}_j^k} c_{\left(\ell^k_1,\cdots,\ell^k_j\right)} \sigma^{\ell^k_1} {\left(\sigma^\prime\right)}^{\ell^k_2} \cdots {\left( \sigma^{(j-1)}\right)}^{\ell^k_j} \right]},
\end{equation}
where $\mathcal{L}_j^k = {\left\{{\left(\ell^k_1,\cdots,\ell^k_j\right)} \in \mathbb{N}^j\,\middle|\,\ell^k_1 + \cdots + \ell_j^k \le \ell^k_2 + 2\ell_3^k + \cdots + (j-1)\ell^k_j\right\}}$ and  $c_{\left(\ell^k_1,\cdots,\ell^k_j\right)}$ are real coefficients. Suppose that~\eqref{eq:v_0-Pn} holds true for some integer $n\ge 1$. Since $v_0^{(n+1)} = \sigma P_n^\prime \circ \varphi^{-1}$, it suffices to prove that $P_{n+1} = \sigma P_n^\prime$ is of the form~\eqref{eq:v_0-Pn}. It holds
\begin{align}
\sigma P_n^\prime =&\ \Cdiff_0^{(n+1)}\sigma^{n+1} + n \Cdiff_0^{(n)} \sigma^n \sigma^\prime \\
&+ \sum_{k = 1}^{n-1} \Cdiff_0^{(k+1)} \sigma^{k+1} {\left[ \sum_{j=1}^{k} \sum_{{\left(\ell^k_1,\cdots,\ell^k_j\right)}\in \mathcal{L}_j^k} c_{\left(\ell^k_1,\cdots,\ell^k_j\right)} \sigma^{\ell^k_1} {\left(\sigma^\prime\right)}^{\ell^k_2} \cdots {\left( \sigma^{(j-1)}\right)}^{\ell^k_j} \right]} \\
&+ \sum_{k = 1}^{n-1} k \Cdiff_0^{(k)} \sigma^k {\left[ \sum_{j=1}^{k} \sum_{{\left(\ell^k_1,\cdots,\ell^k_j\right)}\in \mathcal{L}_j^k} c_{\left(\ell^k_1,\cdots,\ell^k_j\right)} \sigma^{\ell^k_1} {\left(\sigma^\prime\right)}^{\ell^k_2+1} \cdots {\left( \sigma^{(j-1)}\right)}^{\ell^k_j} \right]} \\
&+ \sum_{k = 1}^{n-1} \Cdiff_0^{(k)} \sigma^k {\left[ \sum_{j=1}^{k} \sum_{{\left(\ell^k_1,\cdots,\ell^k_j\right)}\in \mathcal{L}_j^k} \ell_1^k c_{\left(\ell^k_1,\cdots,\ell^k_j\right)} \sigma^{\ell^k_1} {\left(\sigma^\prime\right)}^{\ell^k_2+1} \cdots {\left( \sigma^{(j-1)}\right)}^{\ell^k_j} \right]} + \cdots \\
&+ \sum_{k = 1}^{n-1} \Cdiff_0^{(k)} \sigma^k {\left[ \sum_{j=1}^{k} \sum_{{\left(\ell^k_1,\cdots,\ell^k_j\right)}\in \mathcal{L}_j^k} \ell_j^k c_{\left(\ell^k_1,\cdots,\ell^k_j\right)} \sigma^{\ell^k_1+1} {\left(\sigma^\prime\right)}^{\ell^k_2} \cdots {\left( \sigma^{(j-1)}\right)}^{\ell^k_j-1} \sigma^{(j)} \right]},
\end{align}
which, by rearranging the terms and noting that $(j-1){\left(\ell^k_j-1\right)} + j{\left(\ell_{j+1}^k+1\right)} = (j-1)\ell^k_j + j\ell_{j+1}^k + 1$, reads as~\eqref{eq:v_0-Pn} with $n$ replaced by $n+1$. We next use Assumptions~\ref{hypo:sigma-diffusion} and~\ref{hypo:C0}. The terms~$\Cdiff_0^{(k)}\sigma^k$ are indeed bounded while the terms $\sigma^{\ell^k_1} {\left(\sigma^\prime\right)}^{\ell^k_2} \cdots {\left( \sigma^{(j-1)}\right)}^{\ell^k_j}$ are at most of order $|q|^{\gamma/2{\left(\ell^k_1 + \cdots + \ell^k_j\right)} - {\left(\ell^k_2 + \cdots + (j-1)\ell^k_j\right)}}$ for~$|q| \ge 1$. This concludes the proof since $0\le \gamma\le 1/2$ and $\ell^k_1 + \cdots + \ell_j^k \le \ell^k_2 + 2\ell_3^k + \cdots + (j-1)\ell^k_j$.
\end{proof}

\subsubsection{The function $\Psi$ is smooth and bounded with bounded derivatives}
\label{sec:appendix-Psi-0}

We prove here the following result.
\begin{lemm}
\label{lemm:Psi-derivatives}
The function $\Psi = -\frac{1}{2}\sigma^\prime \circ \varphi^{-1}$ is a smooth and bounded function with bounded derivatives on $\mathbb{R}$.
\end{lemm}
\begin{proof}
By definition, $\Psi$ is smooth, and bounded since $\sigma^\prime$ is bounded. Then, following the proof of Lemma~\ref{lemm:estimate-v-0} and noting that $\Psi = \mathfrak{S}_0 \circ \varphi^{-1}$, with $\mathfrak{S}_0 = -\sigma_0^\prime/2$, is similar to $v_0 = \Cdiff_0 \circ \varphi^{-1}$, the derivative of order $n$ of $\Psi$ reads $\Psi^{(n)} = R_n\circ\varphi^{-1}$, with
\begin{equation}
\label{eq:Psi-Rn}
R_n = \mathfrak{S}_0^{(n)}\sigma^n + \sum_{k = 1}^{n-1} \mathfrak{S}_0^{(k)} \sigma^k {\left[ \sum_{j=1}^{k} \sum_{{\left(\ell^k_1,\cdots,\ell^k_j\right)}\in \mathcal{L}_j^k} c_{\left(\ell^k_1,\cdots,\ell^k_j\right)} \sigma^{\ell^k_1} {\left(\sigma^\prime\right)}^{\ell^k_2} \cdots {\left( \sigma^{(j-1)}\right)}^{\ell^k_j} \right]},
\end{equation}
where $\mathcal{L}_j^k = {\left\{{\left(\ell^k_1,\cdots,\ell^k_j\right)} \in \mathbb{N}^j\,\middle|\,\ell^k_1 + \cdots + \ell_j^k \le \ell^k_2 + 2\ell_3^k + \cdots + (j-1)\ell^k_j\right\}}$ and  $c_{\left(\ell^k_1,\cdots,\ell^k_j\right)}$ are real coefficients. The conclusion follows by noting that $\mathfrak{S}_0^{(k)}\sigma^k$ is bounded for all $k\ge 1$.
\end{proof}

\subsubsection{On the tangent processes of $Y$}
\label{sec:appendix-eta-0}

The following technical result holds true in the framework of Theorem~\ref{thm:regularity} in view of Lemma~\ref{lemm:Psi-derivatives}.

\begin{lemm}
\label{lemm:eta-n}
Let $\eta^{(n)}$ be the tangent process of order $n$ of $Y$, starting from $\eta_0^{(n)} = 0$ for $n\ge 2$ and $\eta^{(1)}_0 = 1$. Then, there exists a non-negative function $\widetilde{K}_n \in \mathcal{C}^0(\mathbb{R}_+)$ such that $0 \le \eta_t^{(n)} \le \widetilde{K}_n(t)$ for all $t \ge 0$. 
\end{lemm}

\begin{proof}
We show by induction that for all $1\le k\le n$, the process $\eta^{(k)}$ is bounded and solution of
\begin{equation}
\label{eq:hypo-tangent}
    d\eta_t^{(k)} = \Psi^\prime(Y_t)\eta^{(k)}_t dt + G_k{\left(Y_t,\eta_t^{(1)},\cdots,\eta_t^{(k-1)}\right)}\, dt,
\end{equation}
where $G_1 = 0$ and, for all $2\le k \le n$,
\begin{equation}
\label{eq:eta-Gn}
G_k{\left(Y_t,\eta_t^{(1)},\cdots,\eta_t^{(k-1)}\right)} = \sum_{p=2}^k \Psi^{(p)}(Y_t) \sum_{{\left(\ell^p_1,\cdots,\ell^p_{k-1}\right)} \in \mathcal{J}_{k-1}^{p}} c_{\ell^{p}_1,\cdots,\ell^{p}_{k-1}} {\left(\eta_t^{(1)}\right)}^{\ell_1^{p}}\cdots {\left(\eta_t^{(k-1)}\right)}^{\ell_{k-1}^{p}},
\end{equation}
is a bounded function, $\mathcal{J}_k^{p} = {\left\{ {\left(\ell_1^{p},\cdots,\ell_k^{p} \right)} \in \mathbb{N}^k\, \middle| \, \ell_1^{p}+\cdots+\ell_k^{p} = p \right\}}$ and $c_{\ell^{p}_1,\cdots,\ell^{p}_{k-1}}$ are real coefficients. The induction basis $k=1$ holds true with by the definition~\eqref{eq:tangent-process} of the tangent process. For the inductive step, let us assume that~\eqref{eq:hypo-tangent} is true for $1\le k\le n$. Then (see~\cite{protter2013stochastic}),
\begin{equation}
    d\eta_t^{(n+1)} = \Psi^\prime(Y_t)\eta^{(n+1)}_t dt + G_{n+1}{\left(Y_t,\eta_t^{(1)},\cdots,\eta_t^{(n)}\right)} dt, \qquad \eta_0^{(n+1)} = 0,
\end{equation}
with
\begin{equation}
\begin{aligned}
&G_{n+1}{\left(Y_t,\eta_t^{(1)},\cdots,\eta_t^{(n)}\right)} = \Psi^{(2)}(Y_t)\eta^{(1)}_t \eta^{(n)}_t \\
& \qquad + \sum_{p=2}^n \Psi^{(p+1)}(Y_t) \eta_t^{(1)} \!\!\! \sum_{{\left(\ell^p_1,\cdots,\ell^p_{n-1}\right)} \in \mathcal{J}_{n-1}^{p}} \!\!\! c_{\left(\ell^{p}_1,\cdots,\ell^{p}_{n-1}\right)} {\left(\eta_t^{(1)}\right)}^{\ell_1^{p}}\cdots {\left(\eta_t^{(n-1)}\right)}^{\ell_{n-1}^{p}}\\
&\qquad + \sum_{p=2}^n \Psi^{(p)}(Y_t) \!\!\! \sum_{{\left(\ell^p_1,\cdots,\ell^p_{n-1}\right)} \in \mathcal{J}_{n-1}^{p}} \!\!\! \ell^{p}_1 c_{\left(\ell^{p}_1,\cdots,\ell^{p}_{n-1}\right)} {\left(\eta_t^{(1)}\right)}^{\ell_1^{p}-1}{\left(\eta_t^{(2)}\right)}^{\ell_2^{p}+1}\cdots {\left(\eta_t^{(n-1)}\right)}^{\ell_{n-1}^{p}} \\
&\qquad + \cdots + \sum_{p=2}^n \Psi^{(p)}(Y_t) \!\!\! \sum_{{\left(\ell^p_1,\cdots,\ell^p_{n-1}\right)} \in \mathcal{J}_{n-1}^{p}} \!\!\! \ell^{p}_{n-1} c_{\left(\ell^{p}_1,\cdots,\ell^{p}_{n-1}\right)} {\left(\eta_t^{(1)}\right)}^{\ell_1^{p}}\cdots {\left(\eta_t^{(n-1)}\right)}^{\ell_{n-1}^{p}-1}\eta_t^{(n)},
\end{aligned}
\end{equation}
which, by rearranging the terms, reads as~\eqref{eq:eta-Gn}. Then, in view of~\eqref{eq:eta-Gn}, since $\Psi$ has bounded derivatives and we assumed $\eta^{(k)}$ to be bounded for $1 \le k \le n$, the function $G_{n+1}$ is bounded. Using~\eqref{eq:hypo-tangent} (which is in fact a simple ODE with random coefficients), it holds
\begin{equation}
\label{eq:eta-n}
\eta^{(n+1)}_t = \exp{\left( \int_0^t \Psi^\prime(Y_s)\, ds \right)} \int_0^t G_{n+1}(Y_s,\eta_s^{(1)}, \cdots,\eta_s^{(n)}) \exp{\left(-\int_0^s \Psi^\prime(Y_u)\, du \right)}\,ds.
\end{equation}
This equality allows to conclude.
\end{proof}

\subsection{Proofs on the relation between Becker--D\"oring equations and their Fokker--Planck approximation}
\label{sec:appendix-proofs-CD-FP-Diff}

\subsubsection[Proof]{Proof of Theorem~\ref{thm:Diff2FP}}

The proof mainly consists in rewriting rigorously what we presented in Section~\ref{sec:formal-results}, but with the reverse change of variable. Let us first note that ${\left(G^{-1}\right)}^{\prime} = F \circ G^{-1}$, so that
\begin{equation}
\label{eq:derivatives-Q}
\forall\, (t,x) \in Z_M,\qquad \frac{\partial Q}{\partial x}(t,x) = \frac{F{\left(Q(t,x)\right)}}{F(x)}, \qquad \frac{\partial Q}{\partial t}(t,x) = -F(Q(t,x)).
\end{equation}
Then, by the chain rule,
\begin{equation}
\frac{\partial \CFP}{\partial x}(t,x) = \frac{F{\left(Q(t,x)\right)}}{F(x)}\frac{\partial \Cdiff}{\partial q}(t,Q(t,x)),
\end{equation}
and
\begin{equation}
\begin{aligned}
\frac{\partial^2 \CFP}{\partial x^2}(t,x) &= \frac{F^2{\left(Q(t,x)\right)}}{F^2(x)} \frac{\partial^2 \Cdiff}{\partial q^2}(t,Q(t,x)) + \frac{F{\left(Q(t,x)\right)}(F^\prime{\left(Q(t,x)\right)}-F^\prime(x))}{F^2(x)} \frac{\partial \Cdiff}{\partial q}(t,Q(t,x)).
\end{aligned}
\end{equation}
We also have
\begin{equation}
\label{eq:der-FC}
\begin{aligned}
\frac{\partial (F\CFP)}{\partial x}(t,x) &= F{\left(x\right)} \frac{\partial \CFP}{\partial x}(t,x) + F^\prime{\left(x\right)} \CFP(t,x) = F{\left(Q(t,x)\right)} \frac{\partial \Cdiff}{\partial q}(t,Q(t,x)) + F^\prime{\left(x\right)} \Cdiff(t,Q(t,x)),
\end{aligned}
\end{equation}
and
\begin{align}
\label{eq:der-DC}
\frac{\partial^2 (D\CFP)}{\partial x^2}(t,x) &= D{\left(x\right)}\frac{\partial^2\CFP}{\partial x^2}(t,x) + 2 D^\prime{\left(x\right)}\frac{\partial\CFP}{\partial x}(t,x) + D^{\prime\prime}{\left(x\right)} \CFP(t,x) \\
& = D(Q(t,x))\frac{\partial^2\Cdiff}{\partial q^2}(t,Q(t,x)) + {\left(D(x) \frac{F^2{\left(Q(t,x)\right)}}{F^2{\left(x\right)}} - D(Q(t,x)) \right)}\frac{\partial^2\Cdiff}{\partial q^2}(t,Q(t,x))   \\
& \quad + {\left( \frac{D{\left(x\right)}}{F(x)}{\left(F^\prime{\left(Q(t,x)\right)} - F^\prime{\left(x\right)}\right)}  + 2 D^\prime{\left(x\right)} \right)} \frac{F{\left(Q(t,x)\right)}}{F(x)}\frac{\partial \Cdiff}{\partial q}(t,Q(t,x)) + D^{\prime\prime}{\left(x\right)} \Cdiff(t,Q(t,x)).
\end{align}
Taking the time derivative of $\CFP$, we also have
\begin{equation}
\frac{\partial \CFP}{\partial t}(t,x) =  \frac{\partial \Cdiff}{\partial t}(t,Q(t,x)) - F{\left(Q(t,x)\right)}\frac{\partial \Cdiff}{\partial q}(t,Q(t,x)).
\end{equation}
Therefore, it holds
\begin{equation}
\frac{\partial \Cdiff}{\partial t}(t,Q(t,x)) =  \frac{\partial \CFP}{\partial t}(t,x) + \frac{\partial (F\CFP)}{\partial x}(t,x) - F^\prime(x)\Cdiff(t,Q(t,x)),
\end{equation}
and
\begin{equation}
\begin{aligned}
& D(Q(t,x))\frac{\partial^2\Cdiff}{\partial q^2}(t,Q(t,x))  = \frac{\partial^2 (D\CFP)}{\partial x^2}(t,x) - {\left(D(x) \frac{F^2{\left(Q(t,x)\right)}}{F^2{\left(x\right)}} - D(Q(t,x)) \right)}\frac{\partial^2\Cdiff}{\partial q^2}(t,Q(t,x))   \\
& \qquad - {\left( \frac{D{\left(x\right)}}{F(x)}{\left(F^\prime{\left(Q(t,x)\right)} - F^\prime{\left(x\right)}\right)}  + 2 D^\prime{\left(x\right)} \right)} \frac{F{\left(Q(t,x)\right)}}{F(x)}\frac{\partial \Cdiff}{\partial q}(t,Q(t,x)) - D^{\prime\prime}{\left(x\right)} \Cdiff(t,Q(t,x)).
\end{aligned}
\end{equation}
Combining the last two equations and using~\eqref{problem:Cauchy-diffusion} gives us that $\CFP$ is solution of
\begin{equation}
\label{eq:Cdiff2CFP}
\frac{\partial \CFP}{\partial t}(t,x) = -\frac{\partial (F\CFP)}{\partial x}(t,x) + \frac{1}{2}\frac{\partial^2 (D\CFP)}{\partial x^2}(t,x) - R_{\Cdiff}(t,x),
\end{equation}
where
\begin{equation}
\begin{aligned}
R_{\Cdiff}(t,x) &= \frac{1}{2} {\left(D(x)\frac{F^2(Q(t,x))}{F^2(x)}-D(Q(t,x))\right)}\frac{\partial^2\Cdiff}{\partial q^2}(t,Q(t,x)) + {\left(\frac{1}{2}D^{\prime\prime}(x) - F^\prime(x)\right)}\Cdiff(t,Q(t,x)) \\
& \qquad + {\left( \frac{D{\left(x\right)}}{2F(x)}{\left(F^\prime{\left(Q(t,x)\right)} - F^\prime{\left(x\right)}\right)}  +  D^\prime{\left(x\right)} \right)} \frac{F{\left(Q(t,x)\right)}}{F(x)}\frac{\partial \Cdiff}{\partial q}(t,Q(t,x)).
\end{aligned}
\end{equation}
Using Lemma~\ref{lemm:R_F} and Assumption~\ref{hypo:alpha-beta-prime}, we have, as $x\to +\infty$ and for $t \le G(x)$,
\begin{align}
D(x)\frac{F^2(Q(t,x))}{F^2(x)}-D(Q(t,x)) &= D(x){\left[{\left(\frac{F^2(Q(t,x))}{F^2(x)}-1\right)} - {\left( \frac{D(Q(t,x))}{D(x)} - 1\right)}\right]} = \mathrm{O}{\left(x^{2\gamma-1}\right)},
\end{align}
and
\begin{equation}
 {\left( \frac{D{\left(x\right)}}{F(x)}{\left(F^\prime{\left(Q(t,x)\right)} - F^\prime{\left(x\right)}\right)}  + 2 D^\prime{\left(x\right)} \right)} \frac{F{\left(Q(t,x)\right)}}{F(x)} = \mathrm{O}{\left( x^{\gamma-1} \right)},
\end{equation}
as well as $D^{\prime\prime}(x) - F^\prime(x) = \mathrm{O}{\left( x^{\gamma-1} \right)}$. Then, in view of Theorem~\ref{thm:regularity}, there exists a non-negative function $K \in \mathcal{C}^0(\mathbb{R}_+)$, such that, for all $(t,x) \in Z_M$, it holds ${\left| R_{\Cdiff}(t,x) \right|} \le K(t) x^{\gamma-1}$, which concludes the proof.

\subsubsection[Proof]{Proof of Theorem~\ref{thm:Diff2CD}}

Let us first remark that, for all $n \ge n_0$, $\Ccluster_n$ is well defined and smooth by Theorem~\ref{thm:regularity}. Then, for all $(t,x) \in Z_M$, in view of Lemma~\ref{lemm:R_F} and~\eqref{eq:derivatives-Q}, it holds
\begin{equation}
\forall\, (t,n) \in \mathcal{Z}_M, \qquad Q(t,n+1) - Q(t,n) = 1 + R_1(t,n),
\end{equation}
where there is a non-negative function $K_1 \in \mathcal{C}^0(\mathbb{R}_+)$ such that ${\left|R_1(t,n)\right|} \le K_1(t) n^{\gamma-1}$. Next, using once again a Taylor expansion, for all $(t,n) \in  \mathcal{Z}_M$, there exists $\kappa_n \in\ ]Q(t,n),Q(t,n+1)[$ such that
\begin{equation}
\begin{aligned}
\Ccluster_{n+1}(t) - \Ccluster_{n}(t) &= \Cdiff{\left(t,Q(t,n) + 1 + R_1(t,n)\right)} - \Cdiff(t,Q(t,n)) \\
& = {\left(1 + R_1(t,n)\right)}\frac{\partial\Cdiff}{\partial q}{\left(t,Q(t,n)\right)} + \frac{1}{2}{\left(1 + R_1(t,n)\right)}^2\frac{\partial^2 \Cdiff}{\partial q^2}{\left(t,Q(t,n)\right)} +\frac{1}{6}{\left(1 + R_1(t,n)\right)}^3\frac{\partial^3 \Cdiff}{\partial q^3}{\left(t,\kappa_n\right)}.
\end{aligned} 
\end{equation}
Using Theorem~\ref{thm:regularity} and Lemma~\ref{lemm:R_F}, we have
\begin{equation}
\Ccluster_{n+1}(t) - \Ccluster_{n}(t) =  \frac{\partial\Cdiff}{\partial q}{\left(t,Q(t,n)\right)} + \frac{1}{2} \frac{\partial^2 \Cdiff}{\partial q^2}{\left(t,Q(t,n)\right)} + R_2(t,n),
\end{equation}
where there is a non-negative function $K_2 \in \mathcal{C}^0(\mathbb{R}_+)$ such that ${\left|R_2(t,n)\right|} \le K_2(t) n^{-3\gamma/2}$, the dominant terms of the remainder being $R_1 \partial\Cdiff/\partial q$ and $\partial^3\Cdiff/\partial q^3$. Using once again the assumptions on $\alpha$ and $\beta$ (see Assumption~\ref{hypo:alpha-beta-prime}), in particular that $\alpha^\prime(n) = \mathrm{O}(n^{\gamma-1})$ and the fact that $\Cdiff$ is bounded (see Theorem~\ref{thm:regularity}), there is $\xi_n\in [n,n+1]$ such that
\begin{equation}
\begin{aligned}
\alpha_{n+1} \Ccluster_{n+1}(t) - \alpha_n \Ccluster_n(t) &= {\left(\alpha(n) + \alpha^\prime(\xi_n)\right)} \Ccluster_{n+1}(t) - \alpha_n \Ccluster_n(t) \\
& = \alpha(n) \frac{\partial\Cdiff}{\partial q}{\left(t,Q(t,n)\right)} + \frac{1}{2}\alpha(n) \frac{\partial^2 \Cdiff}{\partial q^2}{\left(t,Q(t,n)\right)} + R^\alpha_3(t,n),
\end{aligned}
\end{equation}
where there is a non-negative function $K^\alpha_3 \in \mathcal{C}^0(\mathbb{R}_+)$ such that ${\left|R^\alpha_3(t,n)\right|} \le K_3^\alpha(t)n^{-\gamma/2}$. Similarly,
\begin{equation}
\begin{aligned}
\beta_{n-1} \Ccluster_{n-1}(t) - \beta_n \Ccluster_n(t) &= - \beta(n)  \frac{\partial\Cdiff}{\partial q}{\left(t,Q(t,n)\right)} + \frac{1}{2}\beta(n)\frac{\partial^2 \Cdiff}{\partial q^2}{\left(t,Q(t,n)\right)} + R^\beta_3(t,n),
\end{aligned}
\end{equation}
where there is a non-negative function $K_3^\beta \in \mathcal{C}^0(\mathbb{R}_+)$ such that ${\left|R^\beta_3(t,n)\right|} \le K_3^\beta(t)n^{-\gamma/2}$. Combining these results, and noting that $F(n) = F(Q(t,n)) + F(n)R_{F,1}(t,n)$ (where $R_{F,1}$ is defined in Lemma~\ref{lemm:R_F}) and a similar equality for $D(n)$, we finally obtain
\begin{equation}
\begin{aligned}
\beta_{n-1}\Ccluster_{n-1} C_1 - (\beta_n C_1 + \alpha_n) \Ccluster_n + \alpha_{n+1} \Ccluster_{n+1} = &- F(Q(t,n)) \frac{\partial\Cdiff}{\partial q}{\left(t,Q(t,n)\right)} \\
&+ \frac{1}{2} D(Q(t,n)) \frac{\partial^2 \Cdiff}{\partial q^2}{\left(t,Q(t,n)\right)} + R_3(t,n),
\end{aligned} 
\end{equation}
where there is a non-negative function $K_3 \in \mathcal{C}^0(\mathbb{R}_+)$ such that ${\left|R_3(t,n)\right|} \le K_3(t)n^{-\gamma/2}$. Since $\Cdiff$ is solution of~\eqref{problem:Cauchy-diffusion} and 
\begin{equation}
\begin{aligned}
\frac{d\Ccluster_n}{dt}(t) &= \frac{\partial \Cdiff}{\partial t}{\left(t,Q(t,n)\right)} - F{\left(Q(t,n)\right)} \frac{\partial\Cdiff}{\partial q}{\left(t,Q(t,n)\right)} = - F(Q(t,n)) \frac{\partial\Cdiff}{\partial q}{\left(t,Q(t,n)\right)} + \frac{1}{2} D(Q(t,n)) \frac{\partial^2 \Cdiff}{\partial q^2}{\left(t,Q(t,n)\right)}, \\
\end{aligned}
\end{equation}
it follows that $\Ccluster_n$ satisfies
\begin{equation}
\frac{d\Ccluster_n}{dt} = \beta_{n-1}\Ccluster_{n-1} C_1 - (\beta_n C_1 + \alpha_n) \Ccluster_n + \alpha_{n+1} \Ccluster_{n+1} - R_3(t,n),
\end{equation}
from which the desired conclusion follows.

\section*{Acknowledgements}

The authors would like to thank Benjamin Jourdain for suggesting to use of the Lamperti transform, and Charles-\'{E}douard Br\'{e}hier for fruitful discussions on tangent processes; as well as Bob Kohn for his insights on Becker--D\"oring dynamics. We also would like to thank Fr\'{e}d\'{e}ric Legoll for his help and interest in the problem. We are finally grateful to Thomas Jourdan, Manuel Ath\`{e}nes and Gilles Adjanor for introducing us to this problem.

\bibliographystyle{plain}
\bibliography{biblio}

\begin{thebibliography}{10}

\bibitem{ball1986becker}
J.~M. Ball, J.~Carr, and O.~Penrose.
\newblock The {B}ecker--{D}{\"o}ring cluster equations: basic properties and
  asymptotic behaviour of solutions.
\newblock {\em Commun. Math. Phys.}, 104(4):657--692, 1986.

\bibitem{barbu2007cluster}
A.~Barbu and E.~Clouet.
\newblock Cluster dynamics modeling of materials: advantages and limitations.
\newblock {\em Solid State Phenomena}, 129:51--58, 2007.

\bibitem{BD35}
R.~Becker and W.~D{\"o}ring.
\newblock {Kinetische Behandlung der Keimbildung in {\"u}bers{\"a}ttigten
  D{\"a}mpfen}.
\newblock {\em Annalen der Physik}, 416(8):719--752, 1935.

\bibitem{bell1927partition}
E.T. Bell.
\newblock Partition polynomials.
\newblock {\em Annals of Mathematics}, pages 38--46, 1927.

\bibitem{bortz1975new}
A.B. Bortz, M.H. Kalos, and J.L. Lebowitz.
\newblock A new algorithm for {M}onte {C}arlo simulation of {I}sing spin
  systems.
\newblock {\em Journal of Computational Physics}, 17(1):10--18, 1975.

\bibitem{CDW95}
J.~Carr, D.~B. Duncan, and C.~H. Walshaw.
\newblock Numerical approximation of a metastable system.
\newblock {\em IMA J. Numer. Anal.}, 15(4):505--521, 1995.

\bibitem{cerrai2001second}
S.~Cerrai.
\newblock {\em Second Order PDE's in Finite and Infinite Dimension: A
  Probabilistic Approach}, volume 1762 of {\em Lecture Notes in Mathematics}.
\newblock Springer-Verlag Berlin Heidelberg, 2001.

\bibitem{CGPV02}
J.-F. Collet, T.~Goudon, F.~Poupaud, and A.~Vasseur.
\newblock The {B}eker-{D}\"{o}ring system and its {L}ifshitz-{S}lyozov limit.
\newblock {\em SIAM J. Appl. Math.}, 62(5):1488--1500, 2002.

\bibitem{CS19}
J.~G. Conlon and A.~Schlichting.
\newblock A non-local problem for the {Fokker-Planck} equation related to the
  {Becker-D\"oring} model.
\newblock {\em Discret. Contin. Dyn. Syst.}, 39(4), 2019.

\bibitem{dragomir2003some}
S.S. Dragomir.
\newblock {\em Some Gronwall Type Inequalities and Applications}.
\newblock Nova Science Publishers New York, 2003.

\bibitem{evans}
L.~C. Evans.
\newblock {\em Partial Differential Equations}, volume~19 of {\em Graduate
  Studies in Mathematics}.
\newblock American Mathematical Society, Providence, RI, second edition, 2010.

\bibitem{friedman2008partial}
A.~Friedman.
\newblock {\em Partial Differential Equations of Parabolic Type}.
\newblock Courier Dover Publications, 2008.

\bibitem{friedman2012stochastic}
A.~Friedman.
\newblock {\em Stochastic Differential Equations and Applications}.
\newblock Courier Corporation, 2012.

\bibitem{ghoniem1980numerical}
N.M. Ghoniem and S.~Sharafat.
\newblock A numerical solution to the {F}okker-{P}lanck equation describing the
  evolution of the interstitial loop microstructure during irradiation.
\newblock {\em Journal of Nuclear Materials}, 92(1):121--135, 1980.

\bibitem{golubov2001grouping}
S.I. Golubov, A.M. Ovcharenko, A.V. Barashev, and B.N. Singh.
\newblock Grouping method for the approximate solution of a kinetic equation
  describing the evolution of point-defect clusters.
\newblock {\em Philosophical Magazine A}, 81(3):643--658, 2001.

\bibitem{goodrich1964nucleation}
F.C. Goodrich.
\newblock Nucleation rates and the kinetics of particle growth. {II}. the birth
  and death process.
\newblock {\em Proceedings of the Royal Society of London. Series A.
  Mathematical and Physical Sciences}, 277(1369):167--182, 1964.

\bibitem{hairer2011hormander}
M.~Hairer.
\newblock On {M}alliavin's proof of {H}{\"o}rmander's theorem.
\newblock {\em Bulletin des sciences math{\'e}matiques}, 135(6-7):650--666,
  2011.

\bibitem{HC99}
S.~Hariz and J.~F. Collet.
\newblock A modified version of the {L}ifshitz-{S}lyozov model.
\newblock {\em Appl. Math. Lett.}, 12(1):81--85, 1999.

\bibitem{HY17}
E.~Hingant and R.~Yvinec.
\newblock Deterministic and stochastic {B}ecker-{D}\"{o}ring equations: past
  and recent mathematical developments.
\newblock In {\em Stochastic Processes, Multiscale Modeling, and Numerical
  Methods for Computational Cellular Biology}, pages 175--204. Springer, 2017.

\bibitem{ising1925beitrag}
E.~Ising.
\newblock Beitrag zur {T}heorie des {F}erromagnetismus.
\newblock {\em Zeitschrift f{\"u}r Physik}, 31(1):253--258, 1925.

\bibitem{JN03}
P.~E. Jabin and B.~Niethammer.
\newblock On the rate of convergence to equilibrium in the {Becker-D\"oring}
  equations.
\newblock {\em J. Diff. Eq.}, 191(2):518--543, 2003.

\bibitem{johnson2002curious}
W.P. Johnson.
\newblock The curious history of {F}a{\`a} di {B}runo's formula.
\newblock {\em The American Mathematical Monthly}, 109(3):217--234, March 2002.

\bibitem{jourdan2014efficient}
T.~Jourdan, G.~Bencteux, and G.~Adjanor.
\newblock Efficient simulation of kinetics of radiation induced defects: a
  cluster dynamics approach.
\newblock {\em Journal of Nuclear Materials}, 444(1):298--313, 2014.

\bibitem{jourdan2016accurate}
T.~Jourdan, G.~Stoltz, F.~Legoll, and L.~Monasse.
\newblock An accurate scheme to solve cluster dynamics equations using a
  {F}okker--{P}lanck approach.
\newblock {\em Computer Physics Communications}, 207:170--178, 2016.

\bibitem{kiritani1973analysis}
M.~Kiritani.
\newblock Analysis of the clustering process of supersaturated lattice
  vacancies.
\newblock {\em Journal of the Physical Society of Japan}, 35(1):95--107, 1973.

\bibitem{Laurencot2002becker}
P.~Lauren{\c{c}}ot and S.~Mischler.
\newblock From the {B}ecker--{D}{\"o}ring to the
  {L}ifshitz--{S}lyozov--{W}agner equations.
\newblock {\em J. Statist. Phys.}, 106(5):957--991, 2002.

\bibitem{lebris2008existence}
C.~Le~Bris and P.-L. Lions.
\newblock Existence and uniqueness of solutions to {F}okker--{P}lanck type
  equations with irregular coefficients.
\newblock {\em Communications in Partial Differential Equations},
  33(7):1272--1317, 2008.

\bibitem{pages2006lamperti}
H.~Luschgy and G.~Pag{\`e}s.
\newblock Functional quantization of a class of {B}rownian diffusions: a
  constructive approach.
\newblock {\em Stochastic Processes and their Applications}, 116(2):310--336,
  2006.

\bibitem{niethammer03}
B.~Niethammer.
\newblock On the evolution of large clusters in the {B}ecker--{D}{\"o}ring
  model.
\newblock {\em Journal of Nonlinear Science}, 13(1):115--122, 2003.

\bibitem{niethammer04}
B.~Niethammer.
\newblock A scaling limit of the {B}ecker--{D}{\"o}ring equations in the regime
  of small excess density.
\newblock {\em Journal of Nonlinear Science}, 14(5):453--468, 2004.

\bibitem{norris1998markov}
J.R. Norris.
\newblock {\em Markov Chains}, volume~2 of {\em Cambridge Series in Statistical
  and Probabilistic Mathematics}.
\newblock Cambridge University Press, 1997.

\bibitem{nualart2006malliavin}
D.~Nualart.
\newblock {\em The Malliavin Calculus and Related Topics}, volume 1995 of {\em
  Probability and Its Applications}.
\newblock Springer-Verlag Berlin Heidelberg, 2006.

\bibitem{ovcharenko2003gmic++}
A.M. Ovcharenko, S.I. Golubov, C.H. Woo, and H.~Huang.
\newblock {GMIC}++: grouping method in {C}++: an efficient method to solve
  large number of master equations.
\newblock {\em Computer Physics Communications}, 152(2):208--226, 2003.

\bibitem{pazy2012semigroups}
A.~Pazy.
\newblock {\em Semigroups of Linear Operators and Applications to Partial
  Differential Equations}, volume~44 of {\em Applied Mathematical Science}.
\newblock Springer, New York, NY, 2012.

\bibitem{Penrose97}
O.~Penrose.
\newblock The {B}ecker-{D}\"{o}ring equations at large times and their
  connection with the {LSW} theory of coarsening.
\newblock {\em J. Statist. Phys.}, 89(1-2):305--320, 1997.

\bibitem{PL79}
O.~Penrose and J.~L. Lebowitz.
\newblock Towards a rigorous molecular theory of metastability.
\newblock In E.~W. Montroll and J.~L. Lebowitz, editors, {\em Fluctuation
  Phenomena}, pages 293--340. Elsevier, 1979.

\bibitem{protter2013stochastic}
P.~Protter.
\newblock {\em Stochastic Integration and Differential Equations}, volume~21 of
  {\em Stochastic Modelling and Applied Probability}.
\newblock Springer-Verlag Berlin Heidelberg, 2013.

\bibitem{Slemrod89}
M.~Slemrod.
\newblock Trend to equilibrium in the {B}ecker-{D}\"{o}ring cluster equations.
\newblock {\em Nonlinearity}, 2(3):429--443, 1989.

\bibitem{Slemrod00}
M.~Slemrod.
\newblock The {B}ecker-{D}\"{o}ring equations.
\newblock In {\em Modeling in Applied Sciences}, Model. Simul. Sci. Eng.
  Technol., pages 149--171. Birkh\"{a}user, 2000.

\bibitem{soisson2010atomistic}
F.~Soisson, C.~Becquart, N.~Castin, C.~Domain, L.~Malerba, and E.~Vincent.
\newblock Atomistic {K}inetic {M}onte {C}arlo studies of microchemical
  evolutions driven by diffusion processes under irradiation.
\newblock {\em Journal of Nuclear Materials}, 406(1):55--67, 2010.

\bibitem{Terrier}
P.~Terrier.
\newblock {\em Numerical simulations for predicting the microstructural
  evolution of ferritic alloys. A study of Cluster Dynamics.}
\newblock PhD Thesis, Universit\'e Paris-Est, 2019.

\bibitem{terrier2017coupling}
P.~Terrier, M.~Athènes, T.~Jourdan, G.~Adjanor, and G.~Stoltz.
\newblock Cluster dynamics modelling of materials: A new hybrid
  deterministic/stochastic coupling approach.
\newblock {\em Journal of Computational Physics}, 350:280--295, 2017.

\bibitem{trotter1959product}
H.F. Trotter.
\newblock On the product of semi-groups of operators.
\newblock {\em Proceedings of the American Mathematical Society},
  10(4):545--551, 1959.

\bibitem{Velazquez98}
J.~J.~L. Vel\'{a}zquez.
\newblock The {B}ecker-{D}\"{o}ring equations and the {L}ifshitz-{S}lyozov
  theory of coarsening.
\newblock {\em J. Statist. Phys.}, 92(1-2):195--236, 1998.

\bibitem{voter2007introduction}
A.~Voter.
\newblock Introduction to the kinetic {M}onte {C}arlo method.
\newblock In {\em Radiation Effects in Solids}, NATO Science Series II:
  Mathematics, Physics and Chemistry, pages 1--23. Springer, 2007.

\bibitem{wolfer1977theory}
W.G. Wolfer, L.K. Mansur, and J.A. Sprague.
\newblock Theory of swelling and irradiation creep.
\newblock Technical report, Wisconsin Univ., 1977.

\bibitem{young1966monte}
W.M. Young and E.W. Elcock.
\newblock {M}onte {C}arlo studies of vacancy migration in binary ordered
  alloys: {I}.
\newblock {\em Proceedings of the Physical Society}, 89(3):735, 1966.

\end{thebibliography}

\end{document}